\documentclass{amsart}
\usepackage{amsfonts}
\usepackage{hyperref}

\newtheorem{theorem}{Theorem}
\theoremstyle{plain}

\newtheorem{corollary}{Corollary}

\newtheorem{lemma}{Lemma}

\newtheorem{problem}{Problem}

\newtheorem{remark}{Remark}

\numberwithin{equation}{section}
\input{tcilatex}

\begin{document}
\title[Monotonicity and convexity of a function involving digamma one]{The
monotonicity and convexity of a function involving digamma one and their
applications}
\author{Zhen-Hang Yang}
\address{System Division, Zhejiang Province Electric Power Test and Research
Institute, Hangzhou, Zhejiang, China, 310014}
\email{yzhkm@163.com}
\date{June 3, 2014}
\thanks{This paper is in final form and no version of it will be submitted
for publication elsewhere.}

\begin{abstract}
Let $\mathcal{L}(x,a)$ be defined on $\left( -1,\infty \right) \times \left(
4/15,\infty \right) $ or $\left( 0,\infty \right) \times \left( 1/15,\infty
\right) $ by the formula%
\begin{equation*}
\mathcal{L}(x,a)=\tfrac{1}{90a^{2}+2}\ln \left( x^{2}+x+\tfrac{3a+1}{3}%
\right) +\tfrac{45a^{2}}{90a^{2}+2}\ln \left( x^{2}+x+\allowbreak \tfrac{%
15a-1}{45a}\right) .
\end{equation*}
We investigate the monotonicity and convexity of the function $x\rightarrow
F_{a}\left( x\right) =\psi \left( x+1\right) -\mathcal{L}(x,a)$, where $\psi 
$ denotes the Psi function. And, we determine the best parameter $a$ such
that the inequality $\psi \left( x+1\right) <\left( >\right) \mathcal{L}%
(x,a) $ holds for $x\in \left( -1,\infty \right) $ or $\left( 0,\infty
\right) $, and then, some new and very high accurate sharp bounds for pis
function and harmonic numbers are presented. As applications, we construct a
sequence $\left( l_{n}\left( a\right) \right) $ defined by $l_{n}\left(
a\right) =H_{n}-\mathcal{L}\left( n,a\right) $, which gives extremely
accurate values for $\gamma $.
\end{abstract}

\maketitle

\section{\textsc{Introduction}}

For $x>0$ the classical Euler's gamma function $\Gamma $ and psi (digamma)
function $\psi $ are defined by%
\begin{equation}
\Gamma \left( x\right) =\int_{0}^{\infty }t^{x-1}e^{-t}dt,\text{ \ \ \ \ }%
\psi \left( x\right) =\frac{\Gamma ^{\prime }\left( x\right) }{\Gamma \left(
x\right) },  \label{Gamma}
\end{equation}%
respectively. The derivatives $\psi ^{\prime }$, $\psi ^{\prime \prime }$, $%
\psi ^{\prime \prime \prime }$, ... are known as polygamma functions (see 
\cite{Anderson.125.11.1997}).

The gamma and polygamma functions play a central role in the theory of
special functions and have extensive applications in many branches, such as
mathematical physics, probability, statistics, engineering. In the recent
past, numerous papers have appeared providing various inequalities for gamma
and polygamma functions. A detailed list of references is given in \cite%
{Merkle.0.4.2010}. In addition, some new results can be found in \cite%
{Alzer.285.2003}, \cite{Alzer.16.2004}, \cite{Alzer.20.7.2007}, \cite%
{Alzer.267.2011}, \cite{Batir.6.4.2005.art.103}, \cite{Batir.328.1.2007}, 
\cite{Batir.91.2008}, \cite{Batir.12.1.2009.art.9}, \cite{Batir.2011.inprint}%
, \cite{Berg.1.4.2004}, \cite{Chen.3.1.2011}, \cite{Chen.336.2007}, \cite%
{Chen.217.7.1.2010}, \cite{Clark.1.129.2003}, \cite{Elbert.128.9.2000}, \cite%
{Koumandos.77.2008}, \cite{Mortici.23.1.2010}, \cite{Mortici.32.4.2011}, 
\cite{Qi.39.2.2010}, \cite{Qiu.74.250.2005}, \cite{Vogt.3.5.2002.art.73},
and closely-related references therein.

In particular, we mention the following inequalities proved by Batir \cite[%
Lemma 1.7]{Batir.12.1.2009.art.9}%
\begin{equation}
\ln \left( x+\tfrac{1}{2}\right) <\psi \left( x+1\right) \leq \ln \left(
x+e^{-\gamma }\right)  \label{1.1a}
\end{equation}%
for $x>0$, where $1/2$ and $e^{-\gamma }$ are the best possible constants, $%
\gamma =0.577215664\cdot \cdot \cdot $ is the Euler-Mascheroni constant. Let 
$H_{n}$ denote the harmonic number defined by%
\begin{equation}
H_{n}=\sum_{k=1}^{n}\frac{1}{k},\text{ \ \ }\left( n\in \mathbb{N}\right) .
\label{Harmonic}
\end{equation}%
From (\ref{1.1a}) and the relation $H_{n}=\gamma +\psi \left( n+1\right) $
(see \cite[p.258]{Abramowttz.1972}) it follows that 
\begin{equation}
\gamma +\ln \left( n+\tfrac{1}{2}\right) <H_{n}\leq \gamma +\ln \left(
n+e^{1-\gamma }-1\right)  \label{1.1b}
\end{equation}%
hold for $n\in \mathbb{N}$. In 2011, He showed further in \cite[Corollary 2.2%
]{Batir.2011.inprint} that 
\begin{equation}
\frac{1}{2}\ln \left( x^{2}+x+e^{-2\gamma }\right) \leq \psi \left(
x+1\right) <\frac{1}{2}\ln \left( x^{2}+x+\tfrac{1}{3}\right)  \label{1.2a}
\end{equation}%
is valid for all $x>0$, where $e^{-2\gamma }$ and $1/3$ are the best
possible. As a direct consequence, he obtained for $n\in \mathbb{N}$. 
\begin{equation}
\gamma +\tfrac{1}{2}\ln \left( n^{2}+n+e^{2-2\gamma }-2\right) \leq
H_{n}<\gamma +\tfrac{1}{2}\ln \left( n^{2}+n+\tfrac{1}{3}\right) .
\label{1.2b}
\end{equation}%
Furthermore, for $n\in \mathbb{N}$, the sequence $\left( \sigma _{n}\right) $
defined by 
\begin{equation}
\sigma _{n}=H_{n}-\tfrac{1}{2}\ln \left( n^{2}+n+\tfrac{1}{3}\right)
\label{Batir seq.}
\end{equation}%
is strictly increasing and 
\begin{equation*}
\lim_{n\rightarrow \infty }n^{4}\left( \sigma _{n}-\gamma \right) =-\tfrac{1%
}{180},
\end{equation*}%
which means that the sequence $\left( \sigma _{n}\right) $ converges to $%
\gamma $ like $n^{-4}$.

Motived by the above results, we easily construct a real function $\left(
x,a\right) \rightarrow \mathcal{L}\left( x,a\right) $ defined on $\left(
-1,\infty \right) \times \left( 4/15,\infty \right) $ or $\left( 0,\infty
\right) \times \left( 1/15,\infty \right) $ by the formula 
\begin{equation}
\mathcal{L}(x,a)=\tfrac{1}{90a^{2}+2}\ln \left( x^{2}+x+\tfrac{3a+1}{3}%
\right) +\tfrac{45a^{2}}{90a^{2}+2}\ln \left( x^{2}+x+\tfrac{15a-1}{45a}%
\right) .  \label{L(x,a)}
\end{equation}%
By arithmetic-geometric mean inequality we have%
\begin{eqnarray*}
\mathcal{L}(x,a) &=&\tfrac{1}{2}\ln \left( \left( x^{2}+x+\tfrac{3a+1}{3}%
\right) ^{1/\left( 45a^{2}+1\right) }\left( x^{2}+x+\tfrac{15a-1}{45a}%
\right) ^{45a^{2}/\left( 45a^{2}+1\right) }\right) \\
&\leq &\tfrac{1}{2}\ln \left( \tfrac{1}{45a^{2}+1}\left( x^{2}+x+\tfrac{3a+1%
}{3}\right) +\tfrac{45a^{2}}{45a^{2}+1}\left( x^{2}+x++\tfrac{15a-1}{45a}%
\right) \right) \\
&=&\tfrac{1}{2}\ln \left( x^{2}+x+\tfrac{1}{3}\right) .
\end{eqnarray*}%
From this, if the inequality $\psi \left( x+1\right) \leq \mathcal{L}\left(
x,a\right) $ is valid for $x\in \left( -1,\infty \right) $, then the second
inequality of (\ref{1.2a}) may be refined.

It is clear that 
\begin{equation}
\lim_{x\rightarrow \infty }\left( \psi \left( x+1\right) -\mathcal{L}%
(x,a)\right) =0,  \label{1.4}
\end{equation}%
and employing L'Hospital's rule together with 
\begin{equation}
\psi ^{\prime }\left( x\right) \symbol{126}\frac{1}{x}+\frac{1}{2x^{2}}+%
\frac{1}{6x^{3}}-\frac{1}{30x^{5}}+\frac{1}{42x^{7}}-\frac{1}{30x^{9}}+...%
\text{ (as }x\rightarrow \infty \text{)}  \label{tri-gamma-Exp}
\end{equation}%
(see \cite[pp. 258--260.]{Abramowttz.1972}), we easily get the following
limit relations:

\begin{eqnarray}
\lim_{x\rightarrow \infty }\frac{\psi \left( x+1\right) -\mathcal{L}(x,a)}{%
x^{-6}} &=&-\tfrac{\left( a-\frac{40+3\sqrt{205}}{105}\right) \left( a-\frac{%
40-3\sqrt{205}}{105}\right) }{85050a},  \label{1.3b} \\
\lim_{x\rightarrow \infty }\frac{\psi \left( x+1\right) -\mathcal{L}(x,a_{1})%
}{x^{-8}} &=&-\tfrac{2}{1225},  \label{1.3c}
\end{eqnarray}%
where $a_{1}=\left( 40+3\sqrt{205}\right) /105$. (\ref{1.3b}) and (\ref{1.3c}%
) indicates that $\psi \left( x+1\right) -\mathcal{L}\left( x,a\right) $ and 
$\psi \left( x+1\right) -\mathcal{L}\left( x,a_{1}\right) $ converge to zero
as $x$ tends infinite like $x^{-6}$ and $x^{-8}$, respectively.

The first aim of this paper is to determine the best $a$ such that the
function $F_{a}\left( x\right) =\psi \left( x+1\right) -\mathcal{L}(x,a)$
has monotonicity and convexity properties, which are showed in Section 3.
The second aim is to determine the best $a$ such that the inequality%
\begin{equation}
\psi \left( x+1\right) <\mathcal{L}(x,a)  \label{m}
\end{equation}%
holds for $x\in \left( -1,\infty \right) $ and its reverse holds for $x\in
(0,\infty )$, which yield some new sharp bounds for harmonic numbers $H_{n}$%
, and they are presented in Section 4. In Section 5, as an application, we
construct a sequence $\left( l_{n}\left( a\right) \right) $ defined by $%
l_{n}\left( a\right) =H_{n}-\mathcal{L}\left( n,a\right) $, which gives
extremely accurate values for $\gamma $ and greatly improves some known
results. Lastly, an open problem is posted.

Some complicated algebraic computations are preformed with the aid of
built-in computer algebra system of \emph{Scientific Workplace Version 5.5}.

\section{\textsc{Lemmas}}

\begin{lemma}
\label{Lemma prop. L}Let $\mathcal{L}(x,a)$ be defined by the formula (\ref%
{L(x,a)}).

(i) If $x>-1$ then the function $\mathcal{L}$ is increasing in $a$ on $%
\left( 4/15,\infty \right) $, and 
\begin{equation}
\lim_{a\rightarrow \infty }\mathcal{L}(x,a)=\tfrac{1}{2}\ln \left( x^{2}+x+%
\tfrac{1}{3}\right) .  \label{2.0}
\end{equation}

(ii) If $x>0$ then $a\mapsto \partial \mathcal{L}/\partial x$ is decreasing, 
$a\mapsto \partial ^{2}\mathcal{L}/\partial x^{2}$ is increasing and $%
a\mapsto \partial ^{3}\mathcal{L}/\partial x^{3}$ is decreasing on $\left(
1/15,\infty \right) $.
\end{lemma}

\begin{proof}
(i) Direct partial derivative calculations yield%
\begin{eqnarray}
\frac{\partial \mathcal{L}}{\partial a} &=&\tfrac{45a}{\left(
45a^{2}+1\right) ^{2}}\ln \tfrac{x^{2}+x+\left( 15a-1\right) /\left(
45a\right) }{x^{2}+x+\left( 3a+1\right) /3}  \notag \\
&&+\tfrac{1}{\left( 90a^{2}+2\right) }\tfrac{1}{x^{2}+x+\left( 3a+1\right) /3%
}+\tfrac{1}{90a^{2}+2}\tfrac{1}{x^{2}+x+\left( 15a-1\right) /\left(
45a\right) },  \notag \\
\frac{\partial ^{2}\mathcal{L}}{\partial x\partial a} &=&-\tfrac{45a^{2}+1}{%
4050a^{2}}\tfrac{2x+1}{\left( x^{2}+x+\left( 3a+1\right) /3\right)
^{2}\left( x^{2}+x+\left( 15a-1\right) /\left( 45a\right) \right) ^{2}},
\label{L_ax}
\end{eqnarray}

Clearly, if $x\geq -1/2$, then $\partial ^{2}\mathcal{L}/\partial x\partial
a<0$, which means that $\partial \mathcal{L}/\partial a$ decreases with $x$.
This leads to%
\begin{equation*}
\frac{\partial \mathcal{L}(x,a)}{\partial a}>\lim_{x\rightarrow \infty }%
\frac{\partial \mathcal{L}(x,a)}{\partial a}=0,
\end{equation*}%
that is, $\mathcal{L}\left( x,a\right) $ increases with $a$ on $\left(
4/15,\infty \right) $.

If $-1<x<-1/2$, then $\partial ^{2}\mathcal{L}/\partial x\partial a>0$, that
is, $\partial \mathcal{L}/\partial a$ increases with $x$ on $\left(
-1,-1/2\right) $. Hence, 
\begin{eqnarray*}
\frac{\partial \mathcal{L}(x,a)}{\partial a} &>&\lim_{x\rightarrow -1+}\frac{%
\partial \mathcal{L}(x,a)}{\partial a} \\
&=&\tfrac{45a}{\left( 45a^{2}+1\right) ^{2}}\left( \tfrac{45a^{2}+1}{90a}%
\tfrac{135a^{2}+90a-3}{45a^{2}+12a-1}+\ln \left( \tfrac{\left( 15a-1\right) 
}{15a\left( 3a+1\right) }\right) \right) \\
&:&=\tfrac{45a}{\left( 45a^{2}+1\right) ^{2}}\mathcal{L}_{1}\left( a\right) .
\end{eqnarray*}%
An elementary computation gives 
\begin{equation*}
\mathcal{L}_{1}^{\prime }\left( a\right) =\tfrac{\left( 45a^{2}+1\right) ^{2}%
}{30a^{2}\left( 45a^{2}+12a-1\right) ^{2}}\left( a+\tfrac{\sqrt{6}-1}{15}%
\right) \left( a-\tfrac{\sqrt{6}+1}{15}\right) >0
\end{equation*}%
for $a\in \left( 4/15,\infty \right) $, which indicates that $\mathcal{L}%
_{1}\left( a\right) $ increases with $a$. Consequently, 
\begin{eqnarray*}
\mathcal{L}_{1}\left( a\right) &>&\mathcal{L}_{1}\left( \tfrac{4}{15}\right)
=\left[ \tfrac{45a^{2}+1}{90a}\tfrac{135a^{2}+90a-3}{45a^{2}+12a-1}+\ln
\left( \tfrac{\left( 15a-1\right) }{15a\left( 3a+1\right) }\right) \right]
_{a=4/15} \\
&=&\ln \tfrac{5}{12}+\tfrac{119}{120}>0,
\end{eqnarray*}%
it follows that 
\begin{equation*}
\frac{\partial \mathcal{L}}{\partial a}\left( x,a\right) >\tfrac{45a}{\left(
45a^{2}+1\right) ^{2}}\mathcal{L}_{1}\left( a\right) >0.
\end{equation*}%
Thus, we have $\partial \mathcal{L}/\partial a>0$ for $x\in \left( -1,\infty
\right) $ and $a\in \left( 4/15,\infty \right) $.

(ii) By (\ref{L_ax}) it is clear that $\partial ^{2}\mathcal{L}/\partial
x\partial a=\partial ^{2}\mathcal{L}/\partial a\partial x<0$ if $x>0$, and
so $\partial \mathcal{L}/\partial x$ decreases with $a$.

Partial derivative computations once again give 
\begin{eqnarray}
\mathcal{L}_{x} &=&\tfrac{\partial \mathcal{L}}{\partial x}=\tfrac{1}{%
90a^{2}+2}\tfrac{2x+1}{x^{2}+x+a+1/3}+\tfrac{45a^{2}}{90a^{2}+2}\tfrac{2x+1}{%
x^{2}+x+\left( 15a-1\right) /\left( 45a\right) },  \label{L_x} \\
\mathcal{L}_{xx} &=&\tfrac{\partial ^{2}\mathcal{L}}{\partial x^{2}}=-\tfrac{%
1}{45a^{2}+1}\tfrac{x^{2}+x-a+1/6}{\left( x^{2}+x+a+1/3\right) ^{2}}-\tfrac{%
45a^{2}}{45a^{2}+1}\tfrac{x^{2}+x+1/\left( 45a\right) +1/6}{\left(
x^{2}+x+\left( 15a-1\right) /\left( 45a\right) \right) ^{2}},  \label{L_xx}
\\
\tfrac{\partial ^{3}\mathcal{L}}{\partial x^{2}\partial a} &=&\tfrac{%
45a^{2}+1}{3^{7}5^{3}a^{3}}\tfrac{P(x)}{\left( x^{2}+x+\left( 3a+1\right)
/3\right) ^{3}\left( x^{2}+x+\left( 15a-1\right) /\left( 45a\right) \right)
^{3}},  \notag \\
\mathcal{L}_{xxx} &=&\tfrac{\partial ^{3}\mathcal{L}}{\partial x^{3}}=\tfrac{%
1}{45a^{2}+1}\tfrac{\left( 2x+1\right) \left( x^{2}+x-3a\right) }{\left(
x^{2}+x+a+1/3\right) ^{3}}+\tfrac{45a^{2}}{45a^{2}+1}\tfrac{\left(
2x+1\right) \left( x^{2}+x+1/\left( 15a\right) \right) }{\left(
x^{2}+x+\left( 15a-1\right) /\left( 45a\right) \right) ^{3}},  \label{L_xxx}
\\
\tfrac{\partial ^{4}\mathcal{L}}{\partial x^{3}\partial a} &=&-\tfrac{%
45a^{2}+1}{3^{8}5^{4}}\tfrac{\left( 2x+1\right) \times Q(x)}{\left(
x^{2}+x+a+1/3\right) ^{4}\left( x^{2}+x+\left( 15a-1\right) /\left(
45a\right) \right) ^{4}},  \notag
\end{eqnarray}%
where%
\begin{eqnarray*}
P(x) &=&945ax^{4}+1890ax^{3}+\left( 405a^{2}+1485a-9\right) x^{2} \\
&&+\left( 405a^{2}+540a-9\right) x+\left( 90a^{2}+78a-2\right) ,
\end{eqnarray*}%
\begin{equation*}
Q(x)=\left( 56\,700a^{2}\right) x^{6}+\left( 170\,100a^{2}\right)
x^{5}+b_{4}x^{4}+b_{3}x^{3}+b_{2}x^{2}+b_{1}x+b_{0},
\end{equation*}%
here%
\begin{eqnarray*}
b_{4} &=&\left( 44\,550a^{3}+220\,050a^{2}-990a\right) , \\
b_{3} &=&89\,100a^{3}+156\,600a^{2}-1980a, \\
b_{2} &=&\left( 12\,150a^{4}+67\,500a^{3}+64\,710a^{2}-1500a+6\right) , \\
b_{1} &=&\left( 12\,150a^{4}+22\,950a^{3}+14\,760a^{2}-510a+6\right) , \\
b_{1} &=&\left( 2025a^{4}+2970a^{3}+1530a^{2}-66a+1\right) .
\end{eqnarray*}%
It is easy to verify that all coefficients of $P(x)$ and $Q\left( x\right) $
are positive for $a\in \left( 1/15,\infty \right) $, which leads to $%
\partial ^{3}\mathcal{L}/\partial x^{2}\partial a>0$, $\partial ^{4}\mathcal{%
L}/\partial x^{3}\partial a<0$ for $x>0$, which proves the desired results.
\end{proof}

\begin{remark}
\label{Remark D}By the second assertion in the above lemma, and making use
of mean-value theorem, we see that the following functions%
\begin{eqnarray}
a &\mapsto &\mathcal{L}(x,a)-\mathcal{L}\left( y,a\right) ,  \label{c_y} \\
a &\mapsto &\mathcal{L}_{x}(x,a)-\mathcal{L}_{x}\left( y,a\right) ,  \notag
\\
a &\mapsto &\mathcal{L}_{xx}(x,a)-\mathcal{L}_{xx}\left( y,a\right)  \notag
\end{eqnarray}%
are decreasing, increasing and decreasing on $\left( 1/15,\infty \right) $
for $x>y>0$.
\end{remark}

\begin{lemma}[{\protect\cite[pp. 258--260.]{Abramowttz.1972}}]
\label{Lemma psi(n)}Let $x>0$ and $n\in \mathbb{N}$. Then%
\begin{equation}
\psi ^{\left( n\right) }(x+1)-\psi ^{\left( n\right) }(x)=\frac{\left(
-1\right) ^{n}n!}{x^{n+1}},  \label{2.2}
\end{equation}
\end{lemma}

The following lemma was first used to establish some monotonicity results
for the gamma function \cite{Elbert.128.9.2000}, which also play an
important role in proofs our main results.

\begin{lemma}[\protect\cite{Elbert.128.9.2000}]
\label{Lemma Dfx}Let $f$ be a function defined on an interval $I$ and $%
\lim_{x\rightarrow \infty }f(x)=0$. If $f(x+1)-f(x)>0$ for all $x\in I$,
then $f(x)<0$. If $f(x+1)-f(x)<0$ for all $x\in I$, then $f(x)>0$.
\end{lemma}

\section{Monotonicity and convexity}

\begin{theorem}
\label{MT F_a1}Let the function $x\rightarrow F_{a}\left( x\right) =\psi
\left( x+1\right) -\mathcal{L}(x,a)\ $be defined on $\left( -1,\infty
\right) $ where $\mathcal{L}(x,a)$ be defined by \ref{L(x,a)}. Then for $%
x>-1 $, we have%
\begin{equation*}
\left( -1\right) ^{n}F_{a_{1}}^{\left( n\right) }\left( x\right) <0,\text{ \
\ \ }n=1,2,3,
\end{equation*}%
where $a_{1}=\left( 40+3\sqrt{205}\right) /105\approx 0.79003$.
\end{theorem}

\begin{proof}
Differentiation yields%
\begin{eqnarray}
F_{a}^{\prime }\left( x\right) &=&\psi ^{\prime }\left( x+1\right) -\mathcal{%
L}_{x}\left( x,a\right) ,  \label{3.0a} \\
F_{a}^{\prime \prime }\left( x\right) &=&\psi ^{\prime \prime }\left(
x+1\right) -\mathcal{L}_{xx}\left( x,a\right) ,  \label{3.0b} \\
F_{a}^{\prime \prime \prime }\left( x\right) &=&\psi ^{\prime \prime \prime
}\left( x+1\right) -\mathcal{L}_{xxx}\left( x,a\right) ,  \label{3.0c}
\end{eqnarray}%
where $\mathcal{L}_{x}\left( x,a\right) $ $\mathcal{L}_{xx}\left( x,a\right) 
$ and $\mathcal{L}_{xxx}\left( x,a\right) $ are given by (\ref{L_x}), (\ref%
{L_xx}) and (\ref{L_xxx}), respectively

Clearly, we have 
\begin{equation}
\lim_{x\rightarrow \infty }F_{a}^{\prime }\left( x\right)
=\lim_{x\rightarrow \infty }F_{a}^{\prime \prime }\left( x\right)
=\lim_{x\rightarrow \infty }F_{a}^{\prime \prime \prime }\left( x\right) =0.
\label{3.0d}
\end{equation}

From the relation (\ref{2.2}), it is deduced that 
\begin{eqnarray*}
&&F_{a}^{\prime }\left( x+1\right) -F_{a}^{\prime }\left( x\right) =\psi
^{\prime }\left( x+2\right) -\psi ^{\prime }\left( x+1\right) -\mathcal{L}%
_{x}\left( x+1,a\right) +\mathcal{L}_{x}\left( x,a\right) \\
&=&-\tfrac{1}{\left( x+1\right) ^{2}}-\tfrac{2\left( x+1\right) +1}{\left(
90a^{2}+2\right) \left( \left( x+1\right) ^{2}+\left( x+1\right) +\left(
3a+1\right) /3\right) }-\tfrac{45a^{2}\left( 2\left( x+1\right) +1\right) }{%
\left( 90a^{2}+2\right) \left( \left( x+1\right) ^{2}+\left( x+1\right)
+\left( 15a-1\right) /\left( 45a\right) \right) } \\
&&+\tfrac{2x+1}{\left( 90a^{2}+2\right) \left( x^{2}+x+\left( 3a+1\right)
/3\right) }+\tfrac{45a^{2}\left( 2x+1\right) }{\left( 90a^{2}+2\right)
\left( x^{2}+x+\left( 15a-1\right) /\left( 45a\right) \right) },
\end{eqnarray*}%
which, by factoring and simplifying, can be written as 
\begin{equation}
F_{a}^{\prime }\left( x+1\right) -F_{a}^{\prime }\left( x\right) =\frac{%
q(x,a)}{p(x,a)},  \label{3.1}
\end{equation}%
where%
\begin{eqnarray}
q(x,a) &=&\frac{315\left( a+\frac{3\sqrt{205}-40}{105}\right) \left(
a_{1}-a\right) }{2025a}\left( x+1\right) ^{2}-\frac{\left( a+1/3\right)
^{2}\left( a-1/15\right) ^{2}}{9a^{2}}  \label{3.2} \\[0.05in]
p(x,a) &=&\left( x+1\right) ^{2}\left( x^{2}+3x+\left( a+7/3\right) \right)
\left( x^{2}+x+\left( a+1/3\right) \right)  \label{3.3} \\[0.05in]
&&\cdot \left( x^{2}+x+1/3-1/\left( 45a\right) \right) \left(
x^{2}+3x+7/3-1/\left( 45a\right) \right) >0  \notag
\end{eqnarray}

Substituting $a=a_{1}=\left( 40+3\sqrt{205}\right) /105$ into (\ref{3.2})
yields 
\begin{equation*}
q(x,a_{1})=-\tfrac{144}{1225},
\end{equation*}%
and then%
\begin{equation*}
F_{a_{1}}^{\prime }\left( x+1\right) -F_{a_{1}}^{\prime }\left( x\right) =-%
\tfrac{144}{1225}\tfrac{1}{p(x,a_{1})}.
\end{equation*}%
Differentiation again yields%
\begin{eqnarray*}
F_{a_{1}}^{\prime \prime }\left( x+1\right) -F_{a_{1}}^{\prime \prime
}\left( x\right) &=&\left( F_{a_{1}}^{\prime }\left( x+1\right)
-F_{a_{1}}^{\prime }\left( x\right) \right) ^{\prime }=\tfrac{144}{1225}%
\tfrac{p_{x}(x,a_{1})}{p^{2}(x,a_{1})}, \\
F_{a_{1}}^{\prime \prime \prime }\left( x+1\right) -F_{a_{1}}^{\prime \prime
\prime }\left( x\right) &=&\tfrac{144}{1225}\tfrac{%
p_{xx}(x,a_{1})p(x,a_{1})-2p_{x}^{2}(x,a_{1})}{p^{3}(x,a_{1})} \\
&=&-\tfrac{144}{1225}\tfrac{2}{1500625}\tfrac{\left( x+1\right) ^{2}}{%
p^{3}(x,a_{1})}r\left( x\right) ,
\end{eqnarray*}%
where 
\begin{eqnarray*}
w\left( x\right) &=&82534375\left( x+1\right) ^{16}+111903750\left(
x+1\right) ^{14}+117967500\left( x+1\right) ^{12} \\
&&+42925750\left( x+1\right) ^{10}+8270325\left( x+1\right)
^{8}-12773700\left( x+1\right) ^{6} \\
&&+3342880\left( x+1\right) ^{4}-596160\left( x+1\right) ^{2}+62208.
\end{eqnarray*}

Since $p(x,a_{1})$ is clearly positive, if we can prove $w\left( x\right) >0$
for all $x>-1$, then we have $F_{a_{1}}^{\prime \prime \prime }\left(
x+1\right) -F_{a_{1}}^{\prime \prime \prime }\left( x\right) <0$, that is, $%
F_{a_{1}}^{\prime \prime \prime }\left( x\right) >F_{a_{1}}^{\prime \prime
\prime }\left( x+1\right) $, which, by Lemma \ref{Lemma Dfx} together with $%
\lim_{x\rightarrow \infty }F_{a_{1}}^{\prime \prime \prime }\left( x\right)
=0$, yields%
\begin{equation*}
F_{a_{1}}^{\prime \prime \prime }\left( x\right) >\lim_{x\rightarrow \infty
}F_{a_{1}}^{\prime \prime \prime }\left( x\right) =0,
\end{equation*}%
that is, $x\mapsto F_{a_{1}}^{\prime \prime }\left( x\right) $ is strictly
increasing on $\left( -1,\infty \right) $, which results in 
\begin{equation*}
F_{a_{1}}^{\prime \prime }\left( x\right) <\lim_{x\rightarrow \infty
}F_{a_{1}}^{\prime \prime }\left( x\right) =0.
\end{equation*}%
It is concluded that $x\mapsto F_{a_{1}}^{\prime }\left( x\right) $ is
strictly decreasing, and then 
\begin{equation*}
F_{a_{1}}^{\prime }\left( x\right) >\lim_{x\rightarrow \infty
}F_{a_{1}}^{\prime }\left( x\right) =0.
\end{equation*}%
Hence, in order to deduce desired results, we have to show that $w\left(
x\right) >0$ for all $x>-1$. it is enough to prove $w_{8}\left( t\right) >0$
for $t>0$, where 
\begin{eqnarray*}
w_{8}\left( t\right) &=&w\left( \sqrt{t}-1\right)
=82534375t^{8}+111903750t^{7}+117967500t^{6}+42925750t^{5} \\
&&+8270325t^{4}-12773700t^{3}+3342880t^{2}-596160t+62208.
\end{eqnarray*}

Firstly, $w_{8}\left( t\right) >0$ for $t\geq 1/8$. In fact, after replacing 
$t$ by $\left( t_{1}+1/8\right) $ and expanding, we get 
\begin{eqnarray*}
w_{8}\left( t\right) &=&(82\,534\,375t_{1}^{8}+194\,438\,125t_{1}^{7}+\frac{%
4031\,873\,125}{16}t_{1}^{6}+\frac{11\,337\,407\,375}{64}t_{1}^{5} \\
&&+\frac{147\,062\,213\,725}{2048}t_{1}^{4}+\frac{15\,458\,202\,275}{4096}%
t_{1}^{3}+w_{2}\left( t_{1}\right) , \\
w_{2}\left( t_{1}\right) &=&\frac{44\,500\,267\,405}{65\,536}t_{1}^{2}-\frac{%
56\,961\,842\,735}{262\,144}t_{1}+\frac{315\,567\,169\,303}{16\,777\,216}.
\end{eqnarray*}%
Clearly, $w_{8}\left( t\right) -w_{2}\left( t_{1}\right) \geq 0$ due to $%
t_{1}=t-1/8\geq 0$. While $w_{2}\left( t_{1}\right) $ is a quadratic
polynomial, and by an ease check, the discriminant of the quadratic equation
is negative and the coefficient of cubic term is positive, and therefore $%
w_{2}\left( t_{1}\right) >0$. Thus we have $w_{8}\left( t\right) >$ for $%
t\geq 1/8$.

Secondly, we show that $w_{8}\left( t\right) >0$ for $0<t<1/8$. Since the
first five terms of eight degrees polynomial $w_{8}\left( t\right) $ is
clearly positive, it suffices to prove that the last four terms of $%
w_{8}\left( t\right) $, that is, a cubic polynomial 
\begin{equation*}
w_{3}\left( t\right) :=-12773700t^{3}+3342880t^{2}-596160t+62208>0.
\end{equation*}%
As $0<t<1/8$ we have 
\begin{eqnarray*}
w_{3}\left( t\right) &>&-12773700\left( \frac{1}{8}\right)
^{3}+3342880t^{2}-596160t+62208 \\
&=&3342\,880t^{2}-596\,160t+\frac{4769\,199}{128}>0,
\end{eqnarray*}%
where the last inequality holds due to the discriminant of the quadratic
equation is negative and the coefficient of quadratic term is positive.

This completes the proof.
\end{proof}

\begin{remark}
From the proof of Theorem \ref{MT F_a1}, we see that for $a\in \left(
1/15,\infty \right) $%
\begin{equation}
\tfrac{\partial q}{\partial a}=-\tfrac{7\left( 45a^{2}+1\right) }{2025a^{2}}%
\left( x+1\right) ^{2}-\tfrac{2\left( a+1/3\right) \left( a-1/15\right)
\left( 2025a^{2}+45\right) }{18225a^{3}}<0,  \label{dq/da<0}
\end{equation}%
which shows that function $a\rightarrow q(x,a)$ is decreasing on $\left(
1/15,\infty \right) $.
\end{remark}

\begin{remark}
In the proof of Theorem \ref{MT F_a1}, replacing $x$ by $x-1/2$ in (\ref%
{3.0a})--(\ref{3.0c}) and simplifying yield%
\begin{eqnarray*}
F_{a_{1}}^{\prime }\left( x-\tfrac{1}{2}\right) &=&\psi ^{\prime }\left( x+%
\tfrac{1}{2}\right) -\tfrac{20x\left( 84x^{2}+71\right) }{%
1680x^{4}+1560x^{2}+81}, \\
F_{a_{1}}^{\prime \prime }\left( x-\tfrac{1}{2}\right) &=&\psi ^{\prime
\prime }\left( x+\tfrac{1}{2}\right) +\tfrac{20}{3}\tfrac{%
47040x^{6}+75600x^{4}+30116x^{2}-1917}{\left( 560x^{4}+520x^{2}+27\right)
^{2}}, \\
F_{a_{1}}^{\prime \prime \prime }\left( x-\tfrac{1}{2}\right) &=&\psi
^{\prime \prime \prime }\left( x+\tfrac{1}{2}\right) -\tfrac{160}{3}\tfrac{%
x\left( 6585600x^{8}+15052800x^{6}+11696160x^{4}+1820960x^{2}-701703\right) 
}{\left( 560x^{4}+520x^{2}+27\right) ^{3}}
\end{eqnarray*}%
and utilization of Theorem \ref{MT F_a1}, it is acquired directly that for $%
x>-1/2$, the inequalities 
\begin{eqnarray}
\psi ^{\prime }\left( x+\tfrac{1}{2}\right) &>&\tfrac{20x\left(
84x^{2}+71\right) }{1680x^{4}+1560x^{2}+81},  \label{Inequal.2a} \\
\psi ^{\prime \prime }\left( x+\tfrac{1}{2}\right) &<&-\tfrac{20}{3}\tfrac{%
47040x^{6}+75600x^{4}+30116x^{2}-1917}{\left( 560x^{4}+520x^{2}+27\right)
^{2}},  \label{Inequal.2b} \\
\psi ^{\prime \prime \prime }\left( x+\tfrac{1}{2}\right) &>&\tfrac{160}{3}%
\tfrac{x\left(
6585600x^{8}+15052800x^{6}+11696160x^{4}+1820960x^{2}-701703\right) }{\left(
560x^{4}+520x^{2}+27\right) ^{3}}  \label{Inequal.2c}
\end{eqnarray}%
hold true.
\end{remark}

Using Theorem \ref{MT F_a1} with Lemma \ref{Lemma prop. L}, the following
assertion is immediate.

\begin{corollary}
\label{MC F_a1}Let the function $x\rightarrow F_{a}\left( x\right) =\psi
\left( x+1\right) -\mathcal{L}(x,a)\ $be defined on $\left( 0,\infty \right) 
$ where $\mathcal{L}(x,a)$ be defined by \ref{L(x,a)}. Then for $x>0$, $%
F_{a} $ is increasing and concave if and only if $a\geq a_{1}=\left( 40+3%
\sqrt{205}\right) /105\approx 0.79003$.
\end{corollary}

\begin{proof}
The necessity follows from $\lim_{x\rightarrow \infty }x^{7}F_{a}^{\prime
}\left( x\right) \geq 0$ and $\lim_{x\rightarrow \infty }x^{8}F_{a}^{\prime
\prime }\left( x\right) \leq 0$. Using L'Hospital's rule two times to the
relation (\ref{1.3b}) give%
\begin{eqnarray*}
-\tfrac{\left( a-\frac{40+3\sqrt{205}}{105}\right) \left( a-\frac{40-3\sqrt{%
205}}{105}\right) }{85050a} &=&\lim_{x\rightarrow \infty }\frac{\psi
^{\prime }\left( x+1\right) -\mathcal{L}_{x}(x,a)}{-6x^{-7}}\leq 0, \\
-\tfrac{\left( a-\frac{40+3\sqrt{205}}{105}\right) \left( a-\frac{40-3\sqrt{%
205}}{105}\right) }{85050a} &=&\lim_{x\rightarrow \infty }\frac{\psi
^{\prime \prime }\left( x+1\right) -\mathcal{L}_{xx}(x,a)}{\left( -6\right)
\left( -7\right) x^{-8}}\leq 0\text{,}
\end{eqnarray*}%
which yield $a\geq a_{1}$.

By Theorem \ref{MT F_a1} with Lemma \ref{Lemma prop. L}, we obtain that for $%
a\geq a_{1}$,%
\begin{eqnarray*}
F_{a}^{\prime }\left( x\right) &=&\psi ^{\prime }\left( x+1\right) -\mathcal{%
L}_{x}(x,a)\geq \psi ^{\prime }\left( x+1\right) -\mathcal{L}_{x}(x,a_{1})>0,
\\
F_{a}^{\prime \prime }\left( x\right) &=&\psi ^{\prime \prime }\left(
x+1\right) -\mathcal{L}_{xx}(x,a)\leq \psi ^{\prime \prime }\left(
x+1\right) -\mathcal{L}_{xx}(x,a_{1})<0,
\end{eqnarray*}%
which proves the sufficiency.
\end{proof}

\begin{theorem}
\label{MT F_a0'}Let the function $x\rightarrow F_{a}\left( x\right) =\psi
\left( x+1\right) -\mathcal{L}(x,a)$ be defined on $(0,\infty )$ in which $%
\mathcal{L}(x,a)$ be defined by \ref{L(x,a)}. Then $F_{a}$ is decreasing on $%
(0,\infty )$ if and only if $a\in (1/15,a_{0}^{\prime }]$, where%
\begin{equation}
a_{0}^{\prime }=\frac{45-4\pi ^{2}+3\sqrt{4\pi ^{4}-80\pi ^{2}+405}}{%
30\left( \pi ^{2}-9\right) }\approx 0.47053.  \label{a0'}
\end{equation}
\end{theorem}

\begin{proof}
\textbf{Necessity. }The necessity is deduced from%
\begin{eqnarray*}
F_{a}^{\prime }\left( 0\right) &=&\psi ^{\prime }\left( 1\right) -\mathcal{L}%
_{x}(0,a) \\
&=&\frac{\pi ^{2}}{6}-\frac{1}{\left( 90a^{2}+2\right) \left( a+1/3\right) }-%
\frac{45a^{2}}{\left( 90a^{2}+2\right) \left( 15a-1\right) /\left(
45a\right) } \\
&=&\frac{1}{6}\frac{45\left( \pi ^{2}-9\right) a^{2}-3\left( 45-4\pi
^{2}\right) a-\left( \pi ^{2}-9\right) }{\left( 3a+1\right) \left(
15a-1\right) }\leq 0,
\end{eqnarray*}%
which in combination with $a\in \left( 1/15,\infty \right) $ gives $a\in
(1/15,a_{0}^{\prime }]$.

\textbf{Sufficiency}. Due to Lemma \ref{Lemma prop. L}, $a\mapsto \mathcal{L}%
_{x}(x,a)$ is decreasing on $\left( 1/15,\infty \right) $, to prove the
sufficiency, it is enough to prove $F_{a_{0}^{\prime }}^{\prime }\left(
x\right) <0$ for $x\in (0,\infty )$. We distinguish two cases to prove it.

Case 1: $x\in \left( 1/20,\infty \right) $. From (\ref{3.2}) and (\ref%
{dq/da<0}) and $a_{0}^{\prime }<48/100$, we have%
\begin{eqnarray*}
q(x,a_{0}^{\prime }) &>&q(x,\frac{48}{100})>q(\frac{1}{20},\frac{48}{100}) \\
&=&\left[ \tfrac{315\left( a+\frac{3\sqrt{205}-40}{105}\right) \left( \frac{3%
\sqrt{205}+40}{105}-a\right) }{2025a}\left( \frac{1}{20}+1\right) ^{2}-%
\tfrac{\left( a+1/3\right) ^{2}\left( a-1/15\right) ^{2}}{9a^{2}}\right]
_{a=48/100} \\
&=&\frac{2341\,501}{1312\,200\,000}>0\text{,}
\end{eqnarray*}%
which in conjunction with (\ref{3.1}) and (\ref{3.3}) yields%
\begin{equation*}
F_{a_{0}^{\prime }}^{\prime }\left( x+1\right) -F_{a_{0}^{\prime }}^{\prime
}\left( x\right) >0.
\end{equation*}%
Hence, we conclude that $F_{a_{0}^{\prime }}^{\prime }\left( x\right)
<\lim_{x\rightarrow \infty }F_{a_{0}^{\prime }}^{\prime }\left( x\right) =0$.

Case 2: $x\in (0,1/20]$. If we show that $F_{a_{0}^{\prime }}^{\prime \prime
}\left( x\right) <0$ for $x\in (0,1/20]$, then we get $F_{a_{0}^{\prime
}}^{\prime }\left( x\right) \leq F_{a_{0}^{\prime }}^{\prime }\left(
0\right) =0$, which proves the desired result. Replacing $x$ by $x+3/2$ in (%
\ref{Inequal.2b}) and using (\ref{2.2}), we have 
\begin{equation*}
\psi ^{\prime \prime }\left( x+1\right) <-\tfrac{20}{3}\tfrac{47040\left(
x+3/2\right) ^{6}+75600\left( x+3/2\right) ^{4}+30116\left( x+3/2\right)
^{2}-1917}{\left( 560\left( x+3/2\right) ^{4}+520\left( x+3/2\right)
^{2}+27\right) ^{2}}-\tfrac{2}{\left( x+1\right) ^{3}}.
\end{equation*}%
Since $a\mapsto \mathcal{L}_{xx}\left( x,a_{0}\right) $ is increasing by
Lemma \ref{Lemma prop. L} and $a_{0}^{\prime }\approx 0.47053>9/20$, we get 
\begin{equation*}
\mathcal{L}_{xx}\left( x,a_{0}^{\prime }\right) >\mathcal{L}_{xx}\left( x,%
\tfrac{9}{20}\right) =-\tfrac{80}{809}\tfrac{x^{2}+x-\frac{17}{60}}{\left(
x^{2}+x+\frac{47}{60}\right) ^{2}}-\tfrac{729}{809}\tfrac{x^{2}+x+\frac{35}{%
162}}{\left( x^{2}+x+\frac{23}{81}\right) ^{2}},
\end{equation*}%
where $\mathcal{L}_{xx}\left( x,a\right) $ is given by (\ref{L_xx}). Thus,
we have 
\begin{eqnarray*}
F_{a_{0}^{\prime }}^{\prime \prime }\left( x\right) &=&\psi ^{\prime \prime
}\left( x+1\right) -\mathcal{L}_{xx}\left( x,a_{0}^{\prime }\right) \\
&<&-\tfrac{20}{3}\tfrac{47040\left( x+1+1/2\right) ^{6}+75600\left(
x+1+1/2\right) ^{4}+30116\left( x+1+1/2\right) ^{2}-1917}{\left( 560\left(
x+1+1/2\right) ^{4}+520\left( x+1+1/2\right) ^{2}+27\right) ^{2}} \\
&&-\tfrac{2}{\left( x+1\right) ^{3}}-\mathcal{L}_{xx}\left( x,\tfrac{9}{20}%
\right) \\
&=&-\tfrac{20}{3}\tfrac{47040\left( x+1+1/2\right) ^{6}+75600\left(
x+1+1/2\right) ^{4}+30116\left( x+1+1/2\right) ^{2}-1917}{\left( 560\left(
x+1+1/2\right) ^{4}+520\left( x+1+1/2\right) ^{2}+27\right) ^{2}}-\tfrac{2}{%
\left( x+1\right) ^{3}} \\
&&+\tfrac{80}{809}\tfrac{x^{2}+x-\frac{17}{60}}{\left( x^{2}+x+\frac{47}{60}%
\right) ^{2}}+\tfrac{729}{809}\tfrac{x^{2}+x+\frac{35}{162}}{\left( x^{2}+x+%
\frac{23}{81}\right) ^{2}}.
\end{eqnarray*}%
Factoring and arranging lead to%
\begin{equation*}
F_{a_{0}^{\prime }}^{\prime \prime }\left( x\right) <\frac{1}{6}\frac{%
P\left( x\right) }{Q\left( x\right) },
\end{equation*}%
where%
\begin{eqnarray*}
P\left( x\right)
&=&9756\,595\,800x^{11}+146\,348\,937\,000x^{10}+1005\,597\,383\,250x^{9} \\
&&+3954\,619\,691\,700x^{8}+9800\,346\,642\,855x^{7}+16\,058\,808\,560%
\,085x^{6} \\
&&+17\,731\,092\,059\,926x^{5}+13\,107\,900\,251\,862x^{4}+6210\,045\,031%
\,977x^{3} \\
&&+1655\,666\,210\,995x^{2}+153\,061\,816\,584x-15\,463\,394\,658,
\end{eqnarray*}%
\begin{eqnarray*}
Q\left( x\right) &=&\left( 60x^{2}+60x+47\right) ^{2}\left(
81x^{2}+81x+23\right) ^{2} \\
&&\times \left( 570x+505x^{2}+210x^{3}+35x^{4}+252\right) ^{2}\left(
x+1\right) ^{3}.
\end{eqnarray*}%
Clear, $Q\left( x\right) >0$ for $x\in (0,1/20]$. While $P\left( x\right) <0$
for $x\in (0,1/20]$ due to $P^{\prime }\left( x\right) >0$ and so 
\begin{equation*}
P\left( x\right) \leq P\left( \tfrac{1}{20}\right) =-\frac{%
2874\,530\,403\,954\,909\,124\,821}{1024\,000\,000\,000}<0,
\end{equation*}%
which leads to $F_{a_{0}^{\prime }}^{\prime \prime }\left( x\right) <0$ for $%
x\in (0,1/20]$.

This completes the proof.
\end{proof}

\begin{theorem}
\label{MT F_a0''}Let the function $x\rightarrow F_{a}\left( x\right) =\psi
\left( x+1\right) -\mathcal{L}(x,a)$ be defined on $(0,\infty )$ in which $%
\mathcal{L}(x,a)$ be defined by \ref{L(x,a)}. Then $F_{a}$ is convex on $%
(0,\infty )$ if and only if $a\in (1/15,a_{0}^{\prime \prime }]$, where $%
a_{0}^{\prime \prime }\approx 0.4321803644583305$ is the unique root of the
equation 
\begin{equation*}
F_{a}^{\prime \prime }\left( 0\right) =\psi ^{\prime \prime }\left( 1\right)
-\mathcal{L}_{xx}\left( 0,a\right) =0
\end{equation*}%
on $\left( 1/15,\infty \right) $, here $\mathcal{L}_{xx}\left( x,a\right) $
is defined by (\ref{L_xx}).
\end{theorem}

\begin{proof}
\textbf{Necessity}. The necessity is deduced from%
\begin{eqnarray*}
F_{a}^{\prime \prime }\left( 0\right) &=&\psi ^{\prime \prime }\left(
1\right) -\mathcal{L}_{xx}\left( 0,a\right) \\
&=&\psi ^{\prime \prime }\left( 1\right) -\frac{a-\frac{1}{6}}{\left(
45a^{2}+1\right) \left( a+\frac{1}{3}\right) ^{2}}+\frac{2025}{2}\frac{%
a^{3}\left( 15a+2\right) }{\left( 15a-1\right) ^{2}\left( 45a^{2}+1\right) }
\\
&=&\psi ^{\prime \prime }\left( 1\right) +\frac{3}{2}\frac{%
2025a^{4}+1620a^{3}+360a^{2}-36a+1}{\left( 3a+1\right) ^{2}\left(
15a-1\right) ^{2}}\geq 0.
\end{eqnarray*}%
Since $a\mapsto \mathcal{L}_{xx}\left( x,a\right) $ is increasing on $\left(
1/15,\infty \right) $, so $a\mapsto F_{a}^{\prime \prime }\left( 0\right) $
is decreasing on the same interval. Note that the facts%
\begin{equation*}
F_{1/3}^{\prime \prime }\left( 0\right) =\frac{171}{64}-2\zeta \left(
3\right) \approx 0.26776>0\text{ and }F_{1/2}^{\prime \prime }\left(
0\right) =\frac{19\,299}{8450}-2\zeta \left( 3\right) \approx -0.12021<0,
\end{equation*}%
we see that the equation $F_{a}^{\prime \prime }\left( 0\right) =0$ has a
unique solution $a_{0}^{\prime \prime }\in \left( 1/3,1/2\right) $ such that 
$F_{a}^{\prime \prime }\left( 0\right) >0$ for $a\in (1/15,a_{0}^{\prime
\prime })$ and $F_{a}^{\prime \prime }\left( 0\right) <0$ for $a\in
(a_{0}^{\prime \prime },\infty )$. Therefore, the solution of the inequality 
$F_{a}^{\prime \prime }\left( 0\right) \geq 0$ is $a\in (1/15,a_{0}^{\prime
\prime }]$. Numerical calculation gives $a_{0}^{\prime \prime }\approx
0.4321803644583305$.

\textbf{Sufficiency. }Now we prove the condition $a\in (1/15,a_{0}^{\prime
\prime }]$ is sufficient for $F_{a}^{\prime \prime }\left( x\right) >0$ to
hold for $x\in \left( 0,\infty \right) $. Due to the increasing property of $%
\mathcal{L}_{xx}\left( x,a\right) $ with respect to $a$ shown by Lemma \ref%
{Lemma prop. L}, we only need to prove $F_{a_{0}^{\prime \prime }}^{\prime
\prime }\left( x\right) >0$ We distinguish to cases:

Case 1: $x\in \left( 3/50,\infty \right) $. Using (\ref{2.2}) and Remark \ref%
{Remark D} together with $a_{0}^{\prime \prime }<9/20$, we have%
\begin{eqnarray*}
F_{a}^{\prime \prime }\left( x+1\right) -F_{a}^{\prime \prime }\left(
x\right) &=&\psi ^{\prime \prime }\left( x+2\right) -\psi ^{\prime \prime
}\left( x+1\right) -\mathcal{L}_{xx}\left( x+1,a\right) +\mathcal{L}%
_{xx}\left( x,a\right) \\
&=&\frac{2}{\left( x+1\right) ^{3}}-\left( \mathcal{L}_{xx}\left(
x+1,9/20\right) -\mathcal{L}_{xx}\left( x,9/20\right) \right) :=-2\frac{%
r\left( x\right) }{s\left( x\right) },
\end{eqnarray*}%
where%
\begin{eqnarray*}
r\left( x\right)
&=&125\,413\,273\,555\,200x^{10}+1254\,132\,735\,552\,000x^{9}+5518\,250%
\,043\,762\,960x^{8} \\
&&+14\,046\,814\,696\,855\,680x^{7}+22\,840\,386\,490\,946\,664x^{6}+24\,664%
\,633\,018\,794\,864x^{5} \\
&&+17\,718\,225\,566\,437\,953x^{4}+8120\,232\,997\,769\,412x^{3}+2081\,281%
\,129\,927\,908x^{2} \\
&&+179\,154\,971\,702\,976x-19\,953\,618\,766\,474,
\end{eqnarray*}%
\begin{eqnarray*}
s\left( x\right) &=&\left( 60x+60x^{2}+47\right) ^{2}\left(
81x+81x^{2}+23\right) ^{2} \\
&&\times \left( 180x+60x^{2}+167\right) ^{2}\left( 243x+81x^{2}+185\right)
^{2}\left( x+1\right) ^{3}.
\end{eqnarray*}

Since $r^{\prime }\left( x\right) >0$, we get%
\begin{equation*}
r\left( x\right) >r\left( 3/50\right) =\frac{1114\,560\,148\,894\,087\,067%
\,992\,508}{3814\,697\,265\,625}>0
\end{equation*}%
for $x\in \left( 3/50,\infty \right) $, while $s\left( x\right) $ is
obviously positive on the same interval. It follows that $F_{a}^{\prime
\prime }\left( x+1\right) -F_{a}^{\prime \prime }\left( x\right) <0$ for $%
x\in \left( 1/10,\infty \right) $, and therefore, $F_{a}^{\prime \prime
}\left( x\right) >\lim_{x\rightarrow \infty }F_{a}^{\prime \prime }\left(
x\right) =0$ for $x\in \left( 1/10,\infty \right) $.

Case 2: $x\in (0,3/50]$. If we show that $F_{a_{0}^{\prime \prime }}^{\prime
\prime \prime }\left( x\right) >0$ for $x\in (0,3/50]$, then we get $%
F_{a_{0}^{\prime \prime }}^{\prime \prime }\left( x\right) \geq
F_{a_{0}^{\prime \prime }}^{\prime \prime }\left( 0\right) =0$, which proves
the desired result. Now replacing $x$ by $x+3/2$ in (\ref{Inequal.2c}) and
using (\ref{2.2}), we have%
\begin{eqnarray*}
\psi ^{\prime \prime \prime }\left( x+1\right) &>&\tfrac{6}{\left(
x+1\right) ^{4}}+\tfrac{160}{3}\left( x+3/2\right) \\
&&\times \tfrac{6585600\left( x+3/2\right) ^{8}+15052800\left( x+3/2\right)
^{6}+11696160\left( x+3/2\right) ^{4}+1820960\left( x+3/2\right) ^{2}-701703%
}{\left( 560\left( x+3/2\right) ^{4}+520\left( x+3/2\right) ^{2}+27\right)
^{3}}.
\end{eqnarray*}%
In view of $a\mapsto \mathcal{L}_{xxx}\left( x,a\right) $ is decreasing by
Lemma \ref{Lemma prop. L} and $a_{0}^{\prime \prime }\approx 0.432180>21/50$%
, we get%
\begin{equation*}
\mathcal{L}_{xxx}\left( x,a_{0}^{\prime \prime }\right) <\mathcal{L}%
_{xxx}\left( x,\frac{21}{50}\right) =\tfrac{3969}{4469}\tfrac{\left(
2x+1\right) \left( x^{2}+x+\frac{10}{63}\right) }{\left( x^{2}+x+\frac{53}{%
189}\right) ^{3}}+\tfrac{500}{4469}\tfrac{\left( 2x+1\right) \left( x^{2}+x-%
\frac{63}{50}\right) }{\left( x^{2}+x+\frac{113}{150}\right) ^{3}}.
\end{equation*}%
Then we have%
\begin{eqnarray*}
F_{a_{0}^{\prime \prime }}^{\prime \prime \prime }\left( x\right) &=&\psi
^{\prime \prime \prime }\left( x+1\right) -\mathcal{L}_{xx}\left(
x,a_{0}^{\prime \prime }\right) \\
&>&\tfrac{6}{\left( x+1\right) ^{4}}+\tfrac{160}{3}\left( x+3/2\right) \times
\\
&&\tfrac{6585600\left( x+3/2\right) ^{8}+15052800\left( x+3/2\right)
^{6}+11696160\left( x+3/2\right) ^{4}+1820960\left( x+3/2\right) ^{2}-701703%
}{\left( 560\left( x+3/2\right) ^{4}+520\left( x+3/2\right) ^{2}+27\right)
^{3}} \\
&&-\tfrac{3969}{4469}\tfrac{\left( 2x+1\right) \left( x^{2}+x+\frac{10}{63}%
\right) }{\left( x^{2}+x+\frac{53}{189}\right) ^{3}}-\tfrac{500}{4469}\tfrac{%
\left( 2x+1\right) \left( x^{2}+x-\frac{63}{50}\right) }{\left( x^{2}+x+%
\frac{113}{150}\right) ^{3}}.
\end{eqnarray*}%
Factoring and arranging lead to%
\begin{equation*}
F_{a_{0}^{\prime \prime }}^{\prime \prime \prime }\left( x\right) >-\frac{1}{%
3}\frac{R\left( x\right) }{S\left( x\right) },
\end{equation*}%
where%
\begin{eqnarray*}
R\left( x\right)
&=&1439\,970\,288\,529\,500\,000x^{19}+33\,839\,301\,780\,443\,250\,000x^{18}
\\
&&+377\,685\,219\,317\,959\,507\,500x^{17}+2619\,038\,198\,507\,995\,293%
\,750x^{16} \\
&&+12\,578\,516\,662\,166\,748\,200\,250x^{15}+44\,394\,499\,254\,715\,419%
\,844\,125x^{14} \\
&&+119\,436\,801\,689\,614\,664\,479\,875x^{13}+250\,817\,342\,412\,016\,626%
\,059\,625x^{12} \\
&&+417\,457\,335\,039\,758\,233\,395\,000x^{11}+555\,642\,395\,442\,917\,892%
\,895\,800x^{10} \\
&&+593\,602\,907\,219\,352\,981\,396\,390x^{9}+508\,233\,654\,389\,427\,279%
\,197\,745x^{8} \\
&&+346\,198\,219\,129\,731\,218\,829\,124x^{7}+184\,849\,155\,080\,550\,188%
\,733\,310x^{6} \\
&&+75\,353\,569\,007\,634\,565\,613\,769x^{5}+22\,380\,430\,314\,381\,942%
\,509\,812x^{4} \\
&&+4414\,609\,088\,286\,249\,144\,994x^{3}+450\,421\,073\,304\,504\,390%
\,873x^{2} \\
&&-4721\,565\,008\,851\,422\,102x-4420\,688\,040\,144\,642\,816,
\end{eqnarray*}%
\begin{eqnarray*}
S\left( x\right) &=&\left( x+1\right) ^{4}\left( 150x^{2}+150x+113\right)
^{3}\left( 189x^{2}+189x+53\right) ^{3} \\
&&\times \left( 35x^{4}+210x^{3}+505x^{2}+570x+252\right) ^{3}.
\end{eqnarray*}%
A simple computation gives $R^{\prime \prime }\left( x\right) >0$ and%
\begin{eqnarray*}
R\left( 0\right) &=&-4420\,688\,040\,144\,642\,816<0, \\
R\left( \tfrac{3}{50}\right) &=&-\tfrac{337\,711\,343\,455\,989\,855\,048%
\,292\,675\,691\,209\,992\,531\,618\,111}{190\,734\,863\,281\,250\,000\,000%
\,000\,000}<0,
\end{eqnarray*}%
which by property of convex functions yield that for $x\in (0,3/50]$,%
\begin{equation*}
R\left( x\right) \leq \frac{3/50-x}{3/50}R\left( 0\right) +\frac{x}{3/50}%
R\left( \tfrac{3}{50}\right) <0.
\end{equation*}%
This in combination with $S\left( x\right) >0$ gives $F_{a_{0}^{\prime
\prime }}^{\prime \prime \prime }\left( x\right) >0$ for $x\in (0,3/50]$.

Thus we complete the proof.
\end{proof}

As a direct consequence of Theorems \ref{MT F_a0'} and \ref{MT F_a0''}, we
have

\begin{corollary}
\label{MC F M-C}Let the function $x\rightarrow F_{a}\left( x\right) =\psi
\left( x+1\right) -\mathcal{L}(x,a)\ $be defined on $\left( 0,\infty \right) 
$ where $\mathcal{L}(x,a)$ be defined by \ref{L(x,a)}. Then for $x>0$, $%
F_{a} $ is decreasing and convex if and only if $a\in \left(
1/15,a_{0}^{\prime \prime }\right) $, where $a_{0}^{\prime \prime }\approx
0.4321803644583305$ is defined in \ref{MT F_a0''}.
\end{corollary}

A easy computation gives%
\begin{eqnarray*}
\mathcal{L}_{x}(x,a_{1}) &=&\left( x+\tfrac{1}{2}\right) \tfrac{x+x^{2}+23/21%
}{x^{4}+2x^{3}+17x^{2}/7+10x/7+12/35}, \\
\mathcal{L}_{x}(x,a_{0}^{\prime }) &=&\left( x+\tfrac{1}{2}\right) \tfrac{%
x^{2}+x+\frac{\pi ^{2}}{15\left( \pi ^{2}-9\right) }}{x^{4}+2x^{3}+\frac{%
7\pi ^{2}-60}{5\left( \pi ^{2}-9\right) }x^{2}+\frac{2\pi ^{2}-15}{5\left(
\pi ^{2}-9\right) }x+\frac{1}{5\left( \pi ^{2}-9\right) }},
\end{eqnarray*}%
and by Corollary \ref{MC F_a1} and Theorem \ref{MT F_a0'} we obtain the
following optimal inequalities.

\begin{corollary}
For $x>0$, the double inequality%
\begin{equation*}
\mathcal{L}_{x}(x,a_{1})<\psi ^{\prime }\left( x+1\right) <\mathcal{L}%
_{x}(x,a_{0}^{\prime })
\end{equation*}%
or equivalently,%
\begin{equation*}
\left( x+\tfrac{1}{2}\right) \tfrac{x+x^{2}+\frac{23}{21}}{x^{4}+2x^{3}+%
\frac{17}{7}x^{2}+\frac{10}{7}x+\frac{12}{35}}<\psi ^{\prime }\left(
x+1\right) <\left( x+\tfrac{1}{2}\right) \tfrac{x^{2}+x+\frac{\pi ^{2}}{%
15\left( \pi ^{2}-9\right) }}{x^{4}+2x^{3}+\frac{7\pi ^{2}-60}{5\left( \pi
^{2}-9\right) }x^{2}+\frac{2\pi ^{2}-15}{5\left( \pi ^{2}-9\right) }x+\frac{1%
}{5\left( \pi ^{2}-9\right) }}
\end{equation*}%
holds with the best constants 
\begin{equation*}
a_{1}=\tfrac{40+3\sqrt{205}}{105}\approx 0.79003\text{ and }a_{0}^{\prime }=%
\tfrac{45-4\pi ^{2}+3\sqrt{4\pi ^{4}-80\pi ^{2}+405}}{30\left( \pi
^{2}-9\right) }\approx 0.47053.
\end{equation*}
\end{corollary}

Similarly, from Corollary \ref{MC F_a1} and Theorem \ref{MT F_a0''} we obtain

\begin{corollary}
For $x>0$, the double inequality%
\begin{equation*}
\mathcal{L}_{xx}(x,a_{0}^{\prime \prime })<\psi ^{\prime \prime }\left(
x+1\right) <\mathcal{L}_{xx}(x,a_{1})
\end{equation*}%
holds with the best constants $a_{0}^{\prime \prime }\approx
0.4321803644583305$ and $a_{1}=\left( 40+3\sqrt{205}\right) /105\approx
0.79003$.

Particularly, taking $a=1/3<a_{0}^{\prime \prime }$, we have%
\begin{eqnarray*}
&&-\tfrac{9}{2}\tfrac{450x^{6}+1350x^{5}+1965x^{4}+1680x^{3}+897x^{2}+282x+38%
}{\left( 3x^{2}+3x+2\right) ^{2}\left( 15x^{2}+15x+4\right) ^{2}} \\
&<&\psi ^{\prime \prime }\left( x+1\right) <-\tfrac{5}{6}\tfrac{\left(
1470x^{6}+4410x^{5}+7875x^{4}+8400x^{3}+5863x^{2}+2398x+346\right) }{\left(
35x^{4}+70x^{3}+85x^{2}+50x+12\right) ^{2}}
\end{eqnarray*}
\end{corollary}

\section{Sharp bounds for Psi function and the harmonic number}

\begin{theorem}
\label{MT Psi<L}Let the function $x\rightarrow \mathcal{L}(x,a)$ be defined
on $\left( -1,\infty \right) $ by (\ref{L(x,a)}) and $a\in \left(
4/15,\infty \right) $. Then inequality 
\begin{equation}
\psi \left( x+1\right) <\mathcal{L}\left( x,a\right)  \label{Psi<L}
\end{equation}%
holds for all $x\in \left( -1,\infty \right) $ if and only if $a\geq
a_{1}=\left( 40+3\sqrt{205}\right) /105\approx 0.79003$.

Moreover, for $x>0$ we have%
\begin{equation}
\mathcal{L}\left( x,a\right) -c_{0}\left( a\right) <\psi \left( x+1\right) <%
\mathcal{L}\left( x,a\right) ,  \label{PB-lr1}
\end{equation}%
where%
\begin{equation}
c_{0}\left( a\right) =\mathcal{L}\left( 0,a\right) +\gamma =\tfrac{1}{%
90a^{2}+2}\ln \tfrac{3a+1}{3}+\tfrac{45a^{2}}{90a^{2}+2}\ln \tfrac{15a-1}{45a%
}+\gamma  \label{c_0}
\end{equation}%
is the best constant, and the lower bound $\mathcal{L}\left( x,a\right)
-c_{0}\left( a\right) $ and upper bound $\mathcal{L}\left( x,a\right) $ are
decreasing and increasing on $\left( a_{1},\infty \right) $, respectively.
\end{theorem}

\begin{proof}
\textbf{Necessity}. If (\ref{Psi<L}) holds, that is, $F_{a}\left( x\right)
=\psi \left( x+1\right) -\mathcal{L}\left( x,a\right) <0$, then by (\ref%
{1.3b}) we have%
\begin{equation*}
\lim_{x\rightarrow \infty }\frac{F_{a}\left( x\right) }{x^{-6}}=-\tfrac{1}{%
85050a}\left( a-\tfrac{40+3\sqrt{205}}{105}\right) \left( a-\tfrac{40-3\sqrt{%
205}}{105}\right) \leq 0.
\end{equation*}%
Solving the inequality for $a$ and noting that $a\in \left( 4/15,\infty
\right) $ yield 
\begin{equation*}
a\geq \tfrac{40+3\sqrt{205}}{105}=a_{1},
\end{equation*}%
which shows that the condition $a\geq a_{1}$ is necessary.

\textbf{Sufficiency}. Suppose that $a\geq a_{1}$. By Theorem \ref{MT F_a1},
it is deduced that 
\begin{equation*}
F_{a_{1}}\left( x\right) =\psi \left( x+1\right) -\mathcal{L}\left(
x,a_{1}\right) <0,
\end{equation*}%
that is, $\psi \left( x+1\right) <\mathcal{L}\left( x,a_{1}\right) $ holds
for all $x\in \left( -1,\infty \right) $. Since the function $a\rightarrow 
\mathcal{L}\left( x,a\right) $ is increasing on $\left( 4/15,\infty \right) $
by Lemma \ref{Lemma prop. L}, it is easy to conclude that for $a\geq a_{1}$,%
\begin{equation*}
\psi \left( x+1\right) <\mathcal{L}\left( x,a_{1}\right) \leq \mathcal{L}%
\left( x,a\right)
\end{equation*}%
holds for all $x\in \left( -1,\infty \right) $, which means that the
condition $a\geq a_{0}$ is sufficient.

Using the monotonicity of $F_{a}\left( x\right) $ and the facts $F_{a}\left(
0\right) =-\gamma -\mathcal{L}\left( 0,a\right) $ and $F_{a}\left( \infty
\right) =0$ gives \ref{PB-lr1}, and the monotonicity of the lower and upper
bounds in $a$ follows from Lemma \ref{Lemma prop. L} and Remark \ref{Remark
D}.

This completes the proof.
\end{proof}

Letting $a=4/5$, $1$, $\infty $ in Theorem \ref{MT Psi<L} we have

\begin{corollary}
The following double inequalities%
\begin{eqnarray*}
&&\frac{5}{298}\ln \left( x^{2}+x+\frac{17}{15}\right) +\frac{72}{149}\ln
\left( x^{2}+x+\frac{11}{36}\right) -c_{0}\left( 4/5\right) \\
&<&\psi \left( x+1\right) <\frac{5}{298}\ln \left( x^{2}+x+\frac{17}{15}%
\right) +\frac{72}{149}\ln \left( x^{2}+x+\frac{11}{36}\right) ,
\end{eqnarray*}%
\begin{eqnarray*}
&&\frac{1}{92}\ln \left( x^{2}+x+\frac{4}{3}\right) +\frac{45}{92}\ln \left(
x^{2}+x+\frac{14}{45}\right) -c_{0}\left( 1\right) \\
&<&\psi \left( x+1\right) <\frac{1}{92}\ln \left( x^{2}+x+\frac{4}{3}\right)
+\frac{45}{92}\ln \left( x^{2}+x+\frac{14}{45}\right) ,
\end{eqnarray*}%
\begin{equation}
\tfrac{1}{2}\ln \left( x^{2}+x+\tfrac{1}{3}\right) -c_{0}\left( \infty
\right) <\psi \left( x+1\right) <\tfrac{1}{2}\ln \left( x^{2}+x+\tfrac{1}{3}%
\right) ,  \label{Yang1}
\end{equation}%
hold true for $x>0$, where 
\begin{eqnarray*}
c_{0}\left( 4/5\right) &=&\gamma +\frac{5}{298}\ln \frac{17}{15}+\frac{72}{%
149}\ln \frac{11}{36}\approx 0.0063957, \\
c_{0}\left( 1\right) &=&\gamma +\frac{1}{92}\ln \frac{4}{3}+\frac{45}{92}\ln 
\frac{14}{45}\approx 0.0092314, \\
c_{0}\left( \infty \right) &=&\gamma +\lim_{a\rightarrow \infty }\left( 
\tfrac{1}{90a^{2}+2}\ln \tfrac{3a+1}{3}+\tfrac{45a^{2}}{90a^{2}+2}\ln \tfrac{%
15a-1}{45a}\right) =\gamma -\tfrac{1}{2}\ln 3\approx 0.027910
\end{eqnarray*}%
are the best constants.
\end{corollary}

\begin{remark}
We easily check that the lower bound in (\ref{Yang1}) is weaker than one in (%
\ref{1.2a}).
\end{remark}

By the relation $\psi \left( n+1\right) =H_{n}-\gamma $ and the fact $%
F_{a}\left( 1\right) =1-\gamma -\mathcal{L}\left( 1,a\right) $, the
inequalities \ref{PB-lr1} can be changed into

\begin{corollary}
\label{Corollary H_na1}Let $\mathcal{L}(x,a)$ be defined by (\ref{L(x,a)})
and $a\geq a_{1}=\left( 40+3\sqrt{205}\right) /105$. Then for all $n\in 
\mathbb{N}$ we have%
\begin{equation}
\mathcal{L}\left( n,a\right) +c_{1}\left( a\right) <H_{n}<\mathcal{L}\left(
n,a\right) +\gamma ,  \label{Inequal.H1}
\end{equation}%
where $c_{1}\left( a\right) =1-\mathcal{L}\left( 1,a\right) $ and $\gamma $
are the best possible. And, the lower bound $\mathcal{L}\left( n,a\right)
+c_{1}\left( a\right) $ and upper bound $\mathcal{L}\left( n,a\right)
+\gamma $ are decreasing and increasing on $\left( a_{1},\infty \right) $,
respectively.
\end{corollary}

\begin{theorem}
\label{MT Psi>L}Let the function $x\rightarrow \mathcal{L}(x,a)$ be defined
on $(0,\infty )$ by (\ref{L(x,a)}). Then inequality 
\begin{equation}
\psi \left( x+1\right) >\mathcal{L}\left( x,a\right)  \label{Psi>L}
\end{equation}%
holds for all $x>0$ if and only if $a\in (1/15,a_{0}]$, where $a_{0}\approx
0.512967071402$ is the unique root of the equation $F_{a}\left( 0\right)
=\psi \left( 1\right) -\mathcal{L}(0,a)=0$ on $(1/15,\infty )$.
\end{theorem}

\begin{proof}
\textbf{Necessity}. The necessity can be derived from $F_{a}\left( 0\right)
=\psi \left( 1\right) -\mathcal{L}(0,a)\geq 0$. Lemma \ref{Lemma prop. L}
shows that the function $\mathcal{L}$ is increasing with $a$ on $\left(
1/15,\infty \right) $, which implies that the function $a\rightarrow
F_{a}\left( 0\right) =\psi (1)-\mathcal{L}(0,a)$ is decreasing on $\left(
1/15,\infty \right) $. Straightforward computations yield%
\begin{equation*}
F_{1/2}\left( 0\right) =4.004\,3\times 10^{-4}>0\text{ \ and \ }%
F_{3/5}\left( 0\right) =-2.372\,7\times 10^{-3}<0,
\end{equation*}%
which reveals that there is a unique point $a_{0}\in \left( 1/15,3/5\right) $
satisfying $F_{a_{0}}\left( 0\right) =0$ such that $F_{a}\left( 0\right) >0$
for $a\in \left( 1/15,a_{0}\right) $ and $F_{a}\left( 0\right) <0$ for $a\in
\left( a_{0},\infty \right) $.

Numerical calculation gives $a_{0}\approx 0.512967071402$, which shows the
necessity.

\textbf{Sufficiency}. By part two of Lemma \ref{Lemma prop. L}, to prove
sufficiency, it is enough to show that%
\begin{equation*}
F_{a_{0}}\left( x\right) =\psi \left( x+1\right) -\mathcal{L}\left(
x,a_{0}\right) \geq 0
\end{equation*}%
holds for all $x>0$. Now we prove it stepwise.

(i) First of all, we prove that $F_{a_{0}}^{\prime }\left( x\right) <0$ for $%
x\geq 1/5$. Since $\lim_{x\rightarrow \infty }F_{a_{0}}^{\prime }\left(
x\right) =0$, from Lemma \ref{Lemma Dfx}, it suffices to show that 
\begin{equation*}
F_{a_{0}}^{\prime }\left( x+1\right) -F_{a_{0}}^{\prime }\left( x\right) =%
\tfrac{q(x,a_{0})}{p(x,a_{0})}>0,
\end{equation*}%
where $q\left( x,a\right) $, $p\left( x,a\right) $ are defined by (\ref{3.2}%
) and (\ref{3.3}), respectively.

Since function $a\rightarrow q(x,a)$ is decreasing on $\left( 1/15,\infty
\right) $ by (\ref{dq/da<0}) and $a_{0}<11/21$, for $x\geq 1/5$, we get%
\begin{eqnarray*}
q(x,a_{0}) &>&q(x,\tfrac{11}{21})=\left[ \tfrac{-315a^{2}+240a+7}{2025a}%
\left( x+1\right) ^{2}-\tfrac{\left( a+1/3\right) ^{2}\left( a-1/15\right)
^{2}}{9a^{2}}\right] _{a=11/21} \\
&=&\tfrac{12}{275}\left( x+1\right) ^{2}-\tfrac{9216}{148\,225}\geq \tfrac{12%
}{275}\left( 1/5+1\right) ^{2}-\tfrac{9216}{148\,225}=\tfrac{2448}{3705\,625}%
>0,
\end{eqnarray*}%
which together with $p(x,a_{0})>0$ yields $F_{a_{0}}^{\prime }\left(
x+1\right) -F_{a_{0}}^{\prime }\left( x\right) >0$. By Lemma \ref{Lemma Dfx}%
, we have $F_{a_{0}}^{\prime }\left( x\right) <0$ for $x\geq 1/5$.

(ii) Secondly, we show that there is a point $x_{0}\in \left( 0,1/5\right) $
such that $F_{a_{0}}^{\prime }\left( x_{0}\right) =0$, and $%
F_{a_{0}}^{\prime }\left( x\right) >0$ if $x\in (0,x_{0})$ and $%
F_{a_{0}}^{\prime }\left( x\right) <0$ if $x\in (x_{0},1/5)$. For this
purpose, it suffice to prove that $F_{a_{0}}^{\prime }$ is decreasing on $%
\left( 0,1/5\right) $ and $F_{a_{0}}^{\prime }\left( 0\right) >0$.

Replacing $x$ by $x+3/2$ in (\ref{Inequal.2b}) and using (\ref{2.2}), we
have 
\begin{equation}
\psi ^{\prime \prime }\left( x+1\right) <-\tfrac{20}{3}\tfrac{47040\left(
x+3/2\right) ^{6}+75600\left( x+3/2\right) ^{4}+30116\left( x+3/2\right)
^{2}-1917}{\left( 560\left( x+3/2\right) ^{4}+520\left( x+3/2\right)
^{2}+27\right) ^{2}}-\tfrac{2}{\left( x+1\right) ^{3}}.  \label{3.4}
\end{equation}%
Since $a\mapsto \mathcal{L}_{xx}\left( x,a_{0}\right) $ is increasing by
Lemma \ref{Lemma prop. L} and $a_{0}>1/2$, we get 
\begin{equation}
\mathcal{L}_{xx}\left( x,a_{0}\right) >\mathcal{L}_{xx}\left( x,\tfrac{1}{2}%
\right) =-\tfrac{4}{49}\tfrac{x^{2}+x-1/3}{\left( x^{2}+x+5/6\right) ^{2}}-%
\tfrac{45}{49}\tfrac{x^{2}+x+19/90}{\left( x^{2}+x+13/45\right) ^{2}},
\label{3.5}
\end{equation}%
where $\mathcal{L}_{xx}\left( x,a\right) $ is given by (\ref{L_xx}).
Utilizations of (\ref{3.4}) and (\ref{3.5}) yield that 
\begin{eqnarray*}
F_{a_{0}}^{\prime \prime }\left( x\right) &=&\psi ^{\prime \prime }\left(
x+1\right) -\mathcal{L}_{xx}\left( x,a_{0}\right) \\
&<&-\tfrac{20}{3}\tfrac{47040\left( x+1+1/2\right) ^{6}+75600\left(
x+1+1/2\right) ^{4}+30116\left( x+1+1/2\right) ^{2}-1917}{\left( 560\left(
x+1+1/2\right) ^{4}+520\left( x+1+1/2\right) ^{2}+27\right) ^{2}} \\
&&-\tfrac{2}{\left( x+1\right) ^{3}}-\mathcal{L}_{xx}\left( x,\tfrac{1}{2}%
\right) \\
&=&-\tfrac{20}{3}\tfrac{47040\left( x+1+1/2\right) ^{6}+75600\left(
x+1+1/2\right) ^{4}+30116\left( x+1+1/2\right) ^{2}-1917}{\left( 560\left(
x+1+1/2\right) ^{4}+520\left( x+1+1/2\right) ^{2}+27\right) ^{2}}-\tfrac{2}{%
\left( x+1\right) ^{3}} \\
&&+\tfrac{4}{49}\tfrac{x^{2}+x-1/3}{\left( x^{2}+x+5/6\right) ^{2}}+\tfrac{45%
}{49}\tfrac{x^{2}+x+19/90}{\left( x^{2}+x+13/45\right) ^{2}}.
\end{eqnarray*}%
Factoring and arranging lead to%
\begin{equation*}
F_{a_{0}}^{\prime \prime }\left( x\right) <\frac{32}{54675}\frac{v\left(
x\right) }{u\left( x\right) },
\end{equation*}%
where 
\begin{eqnarray*}
u(x) &=&\left( 560\left( x+\tfrac{3}{2}\right) ^{4}+520\left( x+\tfrac{3}{2}%
\right) ^{2}+27\right) ^{2}\left( x+1\right) ^{3} \\
&&\times \left( x^{2}+x+\tfrac{13}{45}\right) ^{2}\left( x^{2}+x+\tfrac{5}{6}%
\right) ^{2}>0,
\end{eqnarray*}%
\begin{eqnarray*}
v\left( x\right)
&=&25533900x^{11}+383008500x^{10}+2632741950x^{9}+10267850400x^{8} \\
&&+24970713315x^{7}+39608501505x^{6}+41428932346x^{5}+27795216042x^{4} \\
&&+10641326265x^{3}+1197017371x^{2}-633162120x-192808962.
\end{eqnarray*}%
A simple computation yields $v^{\prime \prime }\left( x\right) >0$ for $x>0$
and%
\begin{equation*}
v\left( 0\right) =-192808962<0\text{ \ and \ }v\left( 1/5\right) =-\frac{%
245738739045744}{1953125}<0,
\end{equation*}%
which by properties of convex functions reveals that for $x\in (0,1/5)$%
\begin{equation*}
v\left( x\right) <\left( 1-5x\right) \times v\left( 0\right) +5x\times
v\left( \tfrac{1}{5}\right) <0.
\end{equation*}%
Thus, $F_{a_{0}}^{\prime \prime }\left( x\right) <0$ for $x\in (0,1/5)$,
which means that $F_{a_{0}}^{\prime }\left( x\right) $ is decreasing on $%
(0,1/5)$.

On the other hand, note that $\mathcal{L}_{x}\left( x,a\right) $ decreases
with $a$ on $\left( 1/15,\infty \right) $ by Lemma \ref{Lemma prop. L} and $%
a_{0}>1/2$, it is derived that%
\begin{eqnarray*}
F_{a_{0}}^{\prime }\left( 0\right) &=&\psi ^{\prime }\left( 1\right) -%
\mathcal{L}_{x}\left( 0,a_{0}\right) >\psi ^{\prime }\left( 1\right) -%
\mathcal{L}_{x}\left( 0,\tfrac{1}{2}\right) \\
&=&\frac{\pi ^{2}}{6}-\frac{213}{130}\approx 6.472\,5\times 10^{-3}>0.
\end{eqnarray*}%
This together with the assertion that $F_{a_{0}}^{\prime }\left( x\right) <0$
for $x\geq 1/5$ proved previously implies that there is a unique point $%
x_{0}\in \left( 0,1/5\right) $ such that $F_{a_{0}}^{\prime }\left(
x_{0}\right) =0$, and $F_{a_{0}}^{\prime }\left( x\right) >0$ for $x\in
(0,x_{0})$ and $F_{a_{0}}^{\prime }\left( x\right) <0$ for $x\in (x_{0},1/5)$%
.

(iii) Finally, from (i) and (ii) we conclude that $F_{a_{0}}^{\prime }\left(
x\right) >0$ for $x\in \left( 0,x_{0}\right) $ and $F_{a_{0}}^{\prime
}\left( x\right) <0$ for $x\in \left( x_{0},\infty \right) $. It follows
that 
\begin{eqnarray*}
F_{a_{0}}\left( x_{0}\right) &>&F_{a_{0}}\left( x\right) >F_{a_{0}}\left(
0\right) =0\text{ for }x\in (0,x_{0}), \\
F_{a_{0}}\left( x_{0}\right) &\geq &F_{a_{0}}\left( x\right)
>\lim_{x\rightarrow \infty }F_{a_{0}}\left( x\right) =0\text{ for }x\in
\lbrack x_{0},\infty ),
\end{eqnarray*}%
where $F_{a_{0}}\left( 0\right) =0$ is due to $a_{0}$ is the unique root of
the equation $F_{a_{0}}\left( 0\right) =0$.

This completes the proof.
\end{proof}

\begin{remark}
From the proof previously we see that $F_{a_{0}}\left( x\right) $ has an
upper bound $F_{a_{0}}\left( x_{0}\right) $ in which $x_{0}$ is a unique
zero point of $F_{a_{0}}^{\prime }\left( x\right) $. Numerical computation
yields 
\begin{equation*}
x_{0}\approx 0.147311876217,\text{ \ \ \ }F_{a_{0}}\left( x_{0}\right)
\approx 0.0004651.
\end{equation*}%
It follows that 
\begin{equation}
\mathcal{L}\left( x,a_{0}\right) \leq \psi \left( x+1\right) \leq \mathcal{L}%
\left( x,a_{0}\right) +0.0004651  \label{Inequal.P3}
\end{equation}%
holds for all $x>0$.
\end{remark}

By the proof previously and Lemma \ref{Lemma prop. L}, we easily obtain that 
$F_{a}$ is decreasing on $[1/5,\infty )$ for $a\in \left( 1/15,a_{0}\right) $%
. It follows from the relation $\psi \left( n+1\right) =H_{n}-\gamma $ with
the fact $F_{a}\left( 1\right) =1-\gamma -\mathcal{L}\left( 1,a\right) $ that

\begin{corollary}
\label{Corollary H_na0}Let $\mathcal{L}(x,a)$ be defined by (\ref{L(x,a)})
and $a_{0}=0.512967071402...$. Then%
\begin{equation}
\mathcal{L}\left( n,a\right) +\gamma <H_{n}<\mathcal{L}\left( n,a\right)
+c_{1}\left( a\right) ,  \label{Inequal.H3}
\end{equation}%
holds for all $n\in \mathbb{N}$ and $a\in \left( 1/15,a_{0}\right) $, where $%
c_{1}\left( a\right) =1-\mathcal{L}\left( 1,a\right) $\ and $\gamma $ are
the best.

In particular, taking $a=1/2$, we have%
\begin{eqnarray*}
&&\frac{2}{49}\ln \left( n^{2}+n+\frac{5}{6}\right) +\frac{45}{98}\ln \left(
n^{2}+n+\frac{13}{45}\right) +\gamma \\
&<&H_{n}<\frac{2}{49}\ln \left( n^{2}+n+\frac{5}{6}\right) +\frac{45}{98}\ln
\left( n^{2}+n+\frac{13}{45}\right) +c_{1}\left( 1/2\right) ,
\end{eqnarray*}%
where%
\begin{equation*}
c_{1}\left( 1/2\right) =1-\frac{2}{49}\ln \frac{17}{6}-\ln \frac{103}{45}%
\approx 0.57726.
\end{equation*}
\end{corollary}

For $a\in (1/15,a_{0}^{\prime }]$, utilizing the decreasing property of $%
F_{a}$ on $\left( 0,\infty \right) $ together with the facts $F_{a}\left(
0\right) =-\gamma -\mathcal{L}\left( 0,a\right) $ and $F_{a}\left( \infty
\right) =0$, we have

\begin{corollary}
For $x>0$, the double inequality%
\begin{equation}
\mathcal{L}\left( x,a\right) <\psi \left( x+1\right) <\mathcal{L}\left(
x,a\right) -c_{0}\left( a\right) ,  \label{PB-lr2}
\end{equation}%
where $c_{0}\left( a\right) $ defined by (\ref{c_0}) is the best constant.
And, the lower bound $\mathcal{L}\left( x,a\right) $ and upper bound $%
\mathcal{L}\left( x,a\right) -c_{0}\left( a\right) $ are respectively
increasing and decreasing on $(1/15,a_{0}^{\prime }]$, where $a_{0}^{\prime
}\approx 0.47053$ is defined by (\ref{a0'}). Particularly, taking $%
a=1/3,4/15,\sqrt{5}/15,1/15^{+}$, we have%
\begin{eqnarray*}
&&\frac{1}{12}\ln \left( x^{2}+x+\frac{2}{3}\right) +\frac{5}{12}\ln \left(
x^{2}+x+\frac{4}{15}\right) \\
&<&\psi \left( x+1\right) <\frac{1}{12}\ln \left( x^{2}+x+\frac{2}{3}\right)
+\frac{5}{12}\ln \left( x^{2}+x+\frac{4}{15}\right) -c_{0}\left( 1/3\right) ,
\end{eqnarray*}%
\begin{eqnarray}
&&\frac{16}{21}\ln \left( x+\frac{1}{2}\right) +\frac{5}{42}\ln \left(
x^{2}+x+\frac{3}{5}\right)  \label{Yang2} \\
&<&\psi \left( x+1\right) <\frac{16}{21}\ln \left( x+\frac{1}{2}\right) +%
\frac{5}{42}\ln \left( x^{2}+x+\frac{3}{5}\right) -c_{0}\left( 4/15\right) ,
\notag
\end{eqnarray}%
\begin{eqnarray}
&&\frac{1}{4}\ln \left( \left( x^{2}+x+\frac{1}{3}\right) ^{2}-\frac{1}{45}%
\right)  \label{Yang3} \\
&<&\psi \left( x+1\right) <\frac{1}{4}\ln \left( \left( x^{2}+x+\frac{1}{3}%
\right) ^{2}-\frac{1}{45}\right) -c_{0}\left( \sqrt{5}/15\right) ,  \notag
\end{eqnarray}%
\begin{equation}
\psi \left( x+1\right) >\frac{1}{12}\ln \left( x^{2}+x\right) +\frac{5}{12}%
\ln \left( x^{2}+x+\frac{2}{5}\right) ,  \label{Yang4}
\end{equation}%
where%
\begin{eqnarray*}
c_{0}\left( 1/3\right) &=&\frac{1}{12}\ln \frac{2}{3}+\frac{5}{12}\ln \frac{4%
}{15}+\gamma \approx -0.0073047, \\
c_{0}\left( 4/15\right) &=&-\frac{8}{21}\ln 4+\frac{5}{42}\ln \frac{3}{5}%
+\gamma \approx -0.011709, \\
c_{0}\left( \sqrt{5}/15\right) &=&\frac{1}{4}\ln \frac{4}{45}+\gamma \approx
-0.027876.
\end{eqnarray*}
\end{corollary}

\begin{remark}
The lower bound for $\psi \left( x+1\right) $ in (\ref{Yang2}) is clearly
stronger than one in (\ref{1.1a}).
\end{remark}

\section{Approximations of Euler's Constant}

The Euler's constant $\gamma $ defined by the limit relation 
\begin{equation*}
\gamma =\lim_{n\rightarrow \infty }\left( H_{n}-\ln n\right) =0.577215664...,
\end{equation*}%
where $H_{n}=\sum_{k=1}^{n}\frac{1}{k}$ is the $n$'th harmonic number, is
one of the most important constants in mathematics, maybe the third next to $%
\pi $ and $e$. It is known that the classical sequence $\gamma
_{n}=H_{n}-\log n$ converges to $\gamma $ very slowly. In fact, Young \cite%
{Young.75.422.1991} proved that the sequence $\left( \gamma _{n}\right) $
converges to $\gamma $ as $n^{-1}$. As a consequence, many mathematicians
tried to define new sequences convergent to this constant with increasingly
higher speed. For example, DeTemple \cite{DeTemple.100.5.1993}, \cite%
{DeTemple.160.1991} introduced a faster sequence $T_{n}=H_{n}-\log (n+1/2)$
and showed that the sequence $\left( T_{n}\right) $ convergent to $\gamma $
like $n^{-2}$. In \cite{Toch.94.8.1989}, Toth defined the sequence%
\begin{equation*}
T_{n}=H_{n}-\ln \left( n+\frac{1}{2}+\frac{1}{24n}\right) ,
\end{equation*}%
which converges to $\gamma $ as $n^{-3}$ proved by Negoi in \cite%
{Negoi.15.1997}. Batir \cite{Batir.6.4.2005.art.103} gave a sequence $\left(
\mu _{n}\right) $ defined by 
\begin{equation*}
\mu _{n}=H_{n}+\frac{1}{2}\ln \frac{e^{1/(n+1)}-1}{n+1/2},
\end{equation*}%
which converges to $\gamma $ as $n^{-3}$ proved by Motici in \cite%
{Mortici.26.1.2010}. In a very recent papers \cite{Merkle.0.4.2010}, \cite%
{Mortici.23.1.2010}, \cite{Mortici.26.1.2010}, \cite{Mortici.3.1.2011}, \cite%
{Mortici.32.4.2011}, \cite{Mortici.57.1.2011}, \cite{Mortici.68.3.2010}, 
\cite{Mortici.8.1.2010}, C. Mortici established serval new sequences
converge to $\gamma $ at faster rate, for instance, he \cite%
{Mortici.8.1.2010} proved that the sequences $(u_{n})$ and $\left(
v_{n}\right) $ defined by, respectively, 
\begin{eqnarray*}
u_{n} &=&H_{n-1}+\frac{1}{\left( 6-2\sqrt{6}\right) n}-\ln \left( n+1/\sqrt{6%
}\right) , \\
v_{n} &=&H_{n-1}+\frac{1}{\left( 6+2\sqrt{6}\right) n}-\ln \left( n-1/\sqrt{6%
}\right)
\end{eqnarray*}%
converge to $\gamma $ as $n^{-3}$, and $\left( \delta _{n}\right) $ converge
to $\gamma $ as $n^{-4}$, where $\delta _{n}$ is the arithmetic mean between 
$u_{n}$ and $v_{n}$. Another sequence converges to $\gamma $ as $n^{-4}$ is $%
\left( \alpha _{n}\right) $ defined by 
\begin{equation*}
\alpha _{n}=H_{n-2}+\frac{23}{24\left( n-1\right) }+\frac{1}{24n}-\ln \left(
n-1/2\right)
\end{equation*}%
given in \cite{Mortici.8.1.2010}. He also gave a sequence that speed of
convergence is $n^{-5}$ in \cite{Mortici.3.1.2011}, however, it is
complicated.

In \cite{Batir.2011.inprint} Batir proposed two better estimations for $%
\gamma $: $\left( \sigma _{n}\right) $ defined by (\ref{Batir seq.}) and $%
\left( \theta _{n}\right) $ defined by 
\begin{equation*}
\theta _{n}=H_{n}+\frac{1}{2}\ln \frac{e^{2/(n+1)}-1}{2n+2},
\end{equation*}%
which also converge to $\gamma $ as $n^{-4}$. Furthermore, he defined $%
\left( \tau _{n}\right) $ as 
\begin{equation*}
\tau _{n}=\frac{\theta _{n}+\sigma _{n}}{2}=H_{n}+\frac{1}{4}\ln \frac{%
e^{2/(n+1)}-1}{2n^{3}+4n^{2}+8n/3+2/3}
\end{equation*}%
and pointed out that the sequence $\left( \tau _{n}\right) $ converges to $%
\gamma $ at rate faster than $n^{-4}$ without proving. Indeed, it is not
difficult to show that 
\begin{equation*}
\lim_{n\rightarrow \infty }n^{5}\left( \tau _{n}-\gamma \right) =-\frac{1}{%
180},
\end{equation*}%
which implies that $\left( \tau _{n}\right) $ converges to $\gamma $ like $%
n^{-5}$. So far, this is one of the best results.

Now we consider our results in this paper. For $a\in (1/105,\infty )$, we
define a sequence with a parameter $\left( l_{n}\left( a\right) \right) $ as 
\begin{equation}
l_{n}\left( a\right) =H_{n}-\tfrac{1}{90a^{2}+2}\ln \left( n^{2}+n+\tfrac{%
3a+1}{3}\right) -\tfrac{45a^{2}}{90a^{2}+2}\ln \left( n^{2}+n+\tfrac{15a-1}{%
45a}\right)  \label{l_n(a)}
\end{equation}%
From (\ref{2.0}) it is easy to get 
\begin{equation*}
l_{n}\left( \infty \right) =H_{n}-\tfrac{1}{2}\ln \left( n^{2}+n+\tfrac{1}{3}%
\right) =\sigma _{n}.
\end{equation*}%
By the relation $\psi \left( n+1\right) =H_{n}-\gamma $ we have 
\begin{equation}
l_{n}\left( a\right) -\gamma =\psi \left( n+1\right) -\mathcal{L}\left(
n,a\right) .  \label{L and l}
\end{equation}%
Thus the limit relations (\ref{1.3b}) and (\ref{1.3c}) can be written as 
\begin{eqnarray*}
\lim_{n\rightarrow \infty }n^{6}\left( l_{n}\left( a\right) -\gamma \right)
&=&-\tfrac{1}{85050a}\left( a-\tfrac{40+3\sqrt{205}}{105}\right) \left( a-%
\tfrac{40-3\sqrt{205}}{105}\right) , \\
\lim_{n\rightarrow \infty }n^{8}\left( l_{n}\left( a_{1}\right) -\gamma
\right) &=&-\tfrac{2}{1225}.
\end{eqnarray*}%
These show that for every $a\in (1/105,\infty )$ the sequence $\left(
l_{n}\left( a\right) \right) $ converges to $\gamma $ as $n^{-6}$ if $a\neq
a_{1}$ and as $n^{-8}$ if $a=a_{1}$.

Not only that, from Corollaries \ref{Corollary H_na0} and \ref{Corollary
H_na1}, we have

\begin{theorem}
Let the sequence $\left( l_{n}\left( a\right) \right) $ be defined by (\ref%
{l_n(a)}) where $a\in (1/15,\infty )$. Then for all $n\in \mathbb{N}$, 
\begin{equation*}
l_{n}\left( a_{1}\right) <\gamma \leq l_{n}\left( a_{0}\right)
\end{equation*}%
hold, where $a_{1}=\left( 40+3\sqrt{205}\right) /105$ is the best constant, $%
a_{0}\approx 0.512967071402$ is defined in Theorem \ref{MT Psi>L}. Also, for
every $n\in \mathbb{N}$, $l_{n}\left( a\right) $ is strictly decreasing with 
$a$ on $\left( 1/15,\infty \right) $.
\end{theorem}

It is clear that our sequence $\left( l_{n}\left( a\right) \right) $ defined
by (\ref{l_n(a)}) gives very accurate values for $\gamma $ than the
approximations mentioned above, which also can be seen in the following
table (for convenience, we take $a=1/2<a_{0}$).%
\begin{equation*}
\begin{tabular}{|l|c|c|c|c|}
\hline
$n$ & $|\delta _{n}-\gamma |$ & $|\tau _{n}-\gamma |$ & $|l_{n}\left(
a_{1}\right) -\gamma |$ & $|l_{n}\left( 1/2\right) -\gamma |$ \\ \hline
\multicolumn{1}{|r|}{$1$} & \multicolumn{1}{|r|}{$1.3945\times 10^{-2}$} & 
\multicolumn{1}{|r|}{$2.825\,1\times 10^{-4}$} & \multicolumn{1}{|r|}{$%
2.178\,4\times 10^{-5}$} & $4.139\,7\times 10^{-5}$ \\ \hline
\multicolumn{1}{|r|}{$2$} & \multicolumn{1}{|r|}{$9.169\,6\times 10^{-4}$} & 
\multicolumn{1}{|r|}{$3.254\,6\times 10^{-5}$} & \multicolumn{1}{|r|}{$%
6.675\,8\times 10^{-7}$} & $3.225\,5\times 10^{-6}$ \\ \hline
\multicolumn{1}{|r|}{$5$} & \multicolumn{1}{|r|}{$2.425\times 10^{-5}$} & 
\multicolumn{1}{|r|}{$8.663\,6\times 10^{-7}$} & \multicolumn{1}{|r|}{$%
1.743\,1\times 10^{-9}$} & $3.771\,7\times 10^{-8}$ \\ \hline
\multicolumn{1}{|r|}{$10$} & \multicolumn{1}{|r|}{$1.524\,6\times 10^{-6}$}
& \multicolumn{1}{|r|}{$3.847\,9\times 10^{-8}$} & \multicolumn{1}{|r|}{$%
1.070\,4\times 10^{-11}$} & $8.271\,1\times 10^{-10}$ \\ \hline
\multicolumn{1}{|r|}{$50$} & \multicolumn{1}{|r|}{$2.444\,2\times 10^{-9}$}
& \multicolumn{1}{|r|}{$1.649\,9\times 10^{-11}$} & \multicolumn{1}{|r|}{$%
3.854\,4\times 10^{-17}$} & $6.833\,8\times 10^{-14}$ \\ \hline
\multicolumn{1}{|r|}{$100$} & \multicolumn{1}{|r|}{$1.527\,7\times 10^{-10}$}
& \multicolumn{1}{|r|}{$5.351\,7\times 10^{-13}$} & \multicolumn{1}{|r|}{$%
1.568\,2\times 10^{-19}$} & $1.100\,9\times 10^{-15}$ \\ \hline
\multicolumn{1}{|r|}{$200$} & \multicolumn{1}{|r|}{$9.548\,6\times 10^{-12}$}
& \multicolumn{1}{|r|}{$1.703\,9\times 10^{-14}$} & \multicolumn{1}{|r|}{$%
6.250\,9\times 10^{-22}$} & $1.746\,4\times 10^{-17}$ \\ \hline
\multicolumn{1}{|r|}{$500$} & \multicolumn{1}{|r|}{$2.444\,4\times 10^{-13}$}
& \multicolumn{1}{|r|}{$1.764\,5\times 10^{-16}$} & \multicolumn{1}{|r|}{$%
4.143\,0\times 10^{-25}$} & $7.218\,1\times 10^{-20}$ \\ \hline
\end{tabular}%
\end{equation*}

Furthermore, from Corollaries \ref{MC F_a1} and \ref{MC F M-C}, we see that
our sequence $\left( l_{n}\left( a\right) \right) $ has well properties,
such as monotonicity and concavity, which are stated as follows.

\begin{theorem}
Let $n\in \mathbb{N}$. Then the sequence $\left( l_{n}\left( a\right)
\right) $ is strictly increasing and concave if $a\geq a_{1}=\left( 40+3%
\sqrt{205}\right) /105$, and decreasing and convex if $a\in \left(
1/15,a_{0}^{\prime \prime }\right) $, where $a_{0}^{\prime \prime }\approx
0.4321803644583305$ is defined in \ref{MT F_a0''}..
\end{theorem}

\section{\textsc{Open Problems}}

Motivated by Theorem \ref{MT F_a1}, we post the following open problems.

\begin{problem}
Let $\mathcal{L}(x,a)$ be defined by \ref{L(x,a)}. Prove that $%
-F_{a_{1}}\left( x\right) =\mathcal{L}(x,a_{1})-\psi \left( x+1\right) $ is
a completely monotonic function on $(-1,\infty )$.
\end{problem}

\end{document}